\documentclass[11pt]{amsart}
\usepackage[colorlinks=true]{hyperref}
\usepackage{amsthm,amsmath,amsxtra,amscd,amssymb,xypic, cleveref,xcolor}
\usepackage[all]{xy}

\newcommand{\Spec}{{\operatorname{Spec}}}

\newcommand{\de}{\delta}
\newcommand{\ep}{\varepsilon}

\newcommand\chars[2]{\left[\begin{smallmatrix}#1\\ #2\end{smallmatrix}\right]}
\newcommand\tc[2]{\theta\chars{#1}{#2}}
\newcommand\tn{\theta_{\rm null}}
\newcommand\vtn{\vartheta_{\rm null}}
\newcommand\tnm{\theta_{\rm m,\, null}}
\newcommand\tnmz{\theta_{\rm m_0,\, null}}
\newcommand\gn{\grad_{\rm null}}
\newcommand\vgn{{\mathcal G}_{\rm null}}
\newcommand\gnm{\grad_{m,\,\rm null}}
\newcommand\gnmz{\grad_{m_0,\,\rm null}}

\newcommand{\CC}{{\mathbb{C}}}
\newcommand{\XX}{{\mathbb{X}}}
\newcommand{\YY}{{\mathbb{Y}}}
\newcommand{\WW}{{\mathbb{W}}}

\newcommand{\HH}{{\mathbb{H}}}
\renewcommand{\AA}{{\mathbb{A}}}
\newcommand{\II}{{\mathbb{I}}}

\newcommand{\RR}{{\mathbb{R}}}
\newcommand{\ZZ}{{\mathbb{Z}}}
\newcommand{\JJ}{{\mathbb{J}}}
\newcommand{\LL}{{\mathbb{L}}}
\newcommand{\HJ}{{\mathbb{HJ}}}

\newcommand{\DD}{{\mathbb{D}}}

\newcommand{\calO}{{\mathcal O}}
\newcommand{\calHJ}{{\mathcal {HJ}}}

\newcommand{\calI}{{\mathcal I}}
\newcommand{\calA}{{\mathcal A}}

\newcommand{\calJ}{{\mathcal J}}

\newcommand{\calE}{{\mathcal E}}

\newcommand{\calX}{{\mathcal X}}

\newcommand{\calD}{{\mathcal D}}
\newcommand{\calR}{{\mathcal R}}

\newcommand{\frakm}{{\mathfrak m}}
\newcommand{\frakM}{{\mathfrak M}}

\newcommand{\op}{\operatorname}

\newcommand{\Sp}{\op{Sp}}
\newcommand{\SL}{\op{SL}}

\newcommand{\Stab}{\op{Stab}}

\newcommand{\grad}{\op{grad}}
\newcommand\diag{\op{diag}}

\newcommand\codim{\op{codim}}

\newcommand\vt{\vartheta}

\theoremstyle{plain}
\newtheorem{thm}{Theorem}
\newtheorem*{mthm*}{Main Theorem} 

\newtheorem{lm}[thm]{Lemma}
\newtheorem{prop}[thm]{Proposition}
\newtheorem{cor}[thm]{Corollary}

\theoremstyle{definition}

\newtheorem{convx}{Convention}
\newenvironment{conv}
{\pushQED{\qed}\convx}
{\popQED\endconvx}
\newtheorem{rem}[thm]{Remark}

\usepackage{color}

\begin{document}

\title[Near the diagonal period matrices]{Moduli of abelian varieties near the locus of products of elliptic curves}

\author[S. Grushevsky]{Samuel Grushevsky}
\address{Department of Mathematics and Simons Center for Geometry and Physics, Stony Brook University, Stony Brook, NY 11794-3651}
\email{sam@math.stonybrook.edu}

\author[R. Salvati Manni]{Riccardo Salvati Manni}
\address{Dipartimento di Matematica, Piazzale Aldo Moro, 2, I-00185 Roma, Italy}
\email{salvati@mat.uniroma1.it}

\thanks{Research of the first author is supported in part by NSF grant DMS-21-01631}

\begin{abstract}
We study various naturally defined subvarieties of the moduli space~$\calA_g$ of complex principally polarized abelian varieties (ppav) in a neighborhood of the locus of products of~$g$~elliptic curves. 

In this neighborhood, we obtain a local description for the locus of hyperelliptic curves, reproving the recent result of Shepherd-Barron~\cite{sbpoincare} that the hyperelliptic locus is locally given by tridiagonal matrices. We further reprove and generalize to arbitrary genus the recent result of Agostini and Chua~\cite{agch} showing that the locus of jacobians of genus 5 curves with a theta-null is an irreducible component of the locus of ppav with a theta-null such that the singular locus of the theta divisor at the corresponding two-torsion point has tangent cone of rank at most 3. We further show that the locus of ppav such that the gradient vanishes, for some odd theta characteristic, locally has codimension~$g$ near the diagonal. Finally, we obtain new results on the locus where the rank of the Hessian of the theta function at a two-torsion point that lies on the theta divisor is equal to 2. 
\end{abstract}

\date{\today}
\maketitle
\section*{Introduction}
Most known constructions of geometrically meaningful subvarieties of the moduli space $\calA_g$ of complex principally polarized abelian varieties (ppav) are either via the Jacobian or Albanese map, or by imposing certain conditions on the theta divisor and its singularities. Two most classical such constructions are of course the locus $\calJ_g^\circ$ of Jacobians of smooth genus $g$ curves, and the theta-null divisor $\vtn$ --- the locus of those ppav that have a vanishing theta constant, or equivalently for which the theta divisor contains an even two-torsion point. Geometrically, one can further consider Jacobians of hyperelliptic curves, intermediate Jacobians of cubic threefolds, etc. Working with the theta divisor, one can impose conditions on the dimension of its singular locus, existence of points or higher multiplicity, or on the local structure of singularities.

In this paper we present a unified approach to determining the local structure and some irreducible components of the subvarieties of~$\calA_g$ defined in these ways, which we apply in various situations. We thus reprove recent results of Shepherd-Barron~\cite{sbpoincare} characterizing the locus of hyperelliptic Jacobians locally near the locus of products of elliptic curves. We show that the locus of ppav with an (odd) two-torsion point of multiplicity three on the theta divisor is smooth, locally of codimension~$g$, as expected, near the diagonal. We further show that the locus of Jacobians with a vanishing theta null is an irreducible component of the locus of ppav with a theta-null such that the Hessian matrix of theta has rank 3 --- thus extending to arbitrary genus the results of Agostini and Chua~\cite{agch} in genus 5, and generalizing our work in genus 4. We further show that the locus of products with an elliptic curve is an irreducible component of the locus where the rank of the Hessian as above is equal to 2.

Our method is inspired by our recent work~\cite{fagrsm} with Hershel Farkas, where the geometric study in the neighborhood of the diagonal allowed us to give an explicit solution to the weak Schottky problem. Our approach consists of investigating the geometry of the various loci near the locus of diagonal period matrices, i.e.~geometrically near $\calA_1\times\dots\times \calA_1\subset\calA_g$, and using Taylor expansions of theta functions and infinitesimal geometry there.

\smallskip
We denote $\HH_g$ the Siegel upper half-space, denote $\calA_g$ the moduli space of ppav, and denote $p:\HH_g\to\calA_g$ the universal covering map, which is the quotient by the action of $\Sp(2g,\ZZ)$. We use 
$\calJ_g\subset\calA_g$ to denote the closure of the locus $\calJ_g^\circ$ of Jacobians of smooth genus $g$ curves, and denote by $\calHJ_g^\circ\subset \calHJ_g\subset\calJ_g$ respectively the locus of Jacobians of hyperelliptic genus $g$ curves and its closure. We denote $\HJ_g^\circ\subset\HJ_g\subset\JJ_g\subset\HH_g$ their respective preimages in the Siegel space, and denote $\HJ_g^\circ$ and $\JJ_g^\circ$ the open subsets of hyperelliptic Jacobians, and Jacobians, of smooth curves.

It is a classical result of Mumford~\cite{mumfordbooktheta2} that the geometrically defined locus $\HJ_g\subset\HH_g$ is defined from the point of view of the geometry of the theta divisor as the locus where a certain configuration of theta constants with characteristics vanishes (we'll review this below in detail). We denote $\calR_g:=(\calA_1\times\calA_{g-1})\cup(\calA_2\times\calA_{g-2})\cup\dots\subset\calA_g$ the locus of decomposable (classically called reducible) ppav --- this of course includes ppav that have more than two factors, which may lie in more than one component of the above union.

Finally, we denote $\calD_g:=\calA_1\times\dots\times\calA_1\subset\calR_g$ the locus of products of elliptic curves, and denote $\DD_g\subset\HH_g$ its preimage in the universal cover. One irreducible component of $\DD_g$ is the locus $\II_g:=\HH_1\times\dots\times \HH_1\subset\HH_g$ of diagonal period matrices. Recently, Shepherd-Barron~\cite{sbpoincare} determined the local structure of $\HJ_g$ near $\II_g$. Our first result is an alternative proof of this side result of his (the main thrust, and the main results of~\cite{sbpoincare}, are on elliptic surfaces, which are beyond the scope of our work):
\begin{thm}[{Shepherd-Barron \cite[Theorem 14.6]{sbpoincare}}] \label{thm:hypsmall}
For {\em every} irreducible component $\XX$ of $\HJ_g$ containing $\II_g\subset\HH_g$, to first order at any point of $\II_g$, $\XX\subset\HH_g$ is defined by the vanishing of the entries $\tau_{ij}$ of the period matrix $\tau$ where $(i,j)$ runs over the set of pairs that are {\em not} edges of the corresponding alkane. In particular, the branch corresponding to the linear alkane equals, to first order, the locus of tridiagonal matrices, i.e.~matrices with non-zero entries only on the main diagonal and on the two diagonals directly above and below it.
\end{thm}
\begin{rem}\label{rem:levelhyp}
We have taken care to phrase the result above carefully on the Siegel space. Note that there is a delicate point here: while $\calHJ_g\subset\calA_g$ is an irreducible algebraic variety, its preimage $\HJ_g\subset\HH_g$ has many irreducible components, a number of which contain $\II_g$. The result above describes the tangent space at a point of $\II_g$ to {\em every} irreducible component of $\HJ_g$ containing~$\II_g$.
\end{rem}
Recall that $\vtn\subset\calA_g$ denotes the locus where some even theta constant~$\tc\ep\de(\tau,0)$ vanishes. Following~\cite{grsmordertwo}, we denote $\vtn^k\subset\vtn$ the locus where the rank of the corresponding Hessian matrix $(\partial_{z_a}\partial_{z_b}\tc\ep\de(\tau,z))|_{z=0}$ is at most $k$. Since the theta function for a block-diagonal period matrix factorizes as the product of theta functions, one immediately sees that $\calA_{g_1}\times\calA_{g_2}\subset\vtn^2$, for any $g_1+g_2=g$ with $g_1,g_2>0$ (see also the explicit expansions in~\Cref{sec:expand}). This is to say that $\calR_g\subset\vtn^2$, and we prove that at least the largest irreducible component of $\calR_g$ is also an irreducible component of $\vtn^2$.
\begin{thm}\label{thm:tn2}
The locus $\calA_1\times\calA_{g-1}$ is an irreducible component of $\vtn^2$.
\end{thm}

From Riemann's theta singularity theorem for Jacobians of curves one deduces the inclusion $\calJ_g\cap\vtn\subset\vtn^3$. Our next result is an alternative proof and a generalization to arbitrary genus of the recent result of Agostini and Chua~\cite{agch}. They prove that in genus 5 there exists an irreducible component of $\JJ_5\cap p^{-1}(\vtn)$ that is also an irreducible component of $p^{-1}(\vtn^3)$ (while we recall that in genus 4 the equality $\calJ_4\cap\vtn=\vtn^3$ was conjectured by H. Farkas~\cite{farkashvanishing} and proven by us in~\cite{grsmgen4}).
\begin{thm}\label{thm:tn3} 
For any genus $g\geq 3$, $\calJ_g\cap \vtn$ is an irreducible component of $\vtn^3$ .
\end{thm}
Finally, we recall that the theta-null divisor has a natural ``odd'' counterpart $\vgn\subset\calA_g$, which is the locus of all ppav such that the gradient $(\partial_{z_i}\tc\ep\de(\tau,z))|_{z=0}$ vanishes, for some odd theta characteristic $\chars\ep\de$. It was conjectured in~\cite{grsmconjectures} that the locus $\vgn$ is purely of codimension $g$ in $\calA_g$. In~\cite{grhu1} this conjecture was proven completely for all~$g\le 5$. We now prove that this also holds for every component intersecting the diagonal or containing the hyperelliptic locus:
\begin{thm}\label{thm:Gg} 
The locus $\vtn\times\calA_1\,\subset\, \calA_{g-1}\times\calA_1 $ is an irreducible component of $\vgn$.

Moreover, any irreducible component of the locus $\vgn$ containing the diagonal is locally smooth along the diagonal, and has codimension $g$ in $\calA_g$.
\end{thm}
\begin{rem}
We note that the first part of the above theorem is the $k=1$ case of~\cite[Thm.~6]{grsmconjectures}.

The dimensionality in the second statement can in fact be reduced to the argument in~\cite{grsmconjectures}, which was proven by a detailed degeneration argument. Indeed, \cite[Prop.~12]{grsmconjectures} describes the boundary of $\vgn$ in the partial toroidal compactification, which turns out to be described geometrically as a union of two components which involve respectively the singular locus of the universal theta divisor in genus~$g-1$ (which always has expected dimension), and of the locus $\vgn$ in genus~$g-1$. In~\cite[Thm.~13]{grsmconjectures} this is used to deduced that the codimension of $\vgn$ is precisely $g$ if all of its components, and all components of such loci for lower genera, intersect the boundary of the partial toroidal compactification. However, what is really used in the proof of that theorem is that for a given irreducible component of $\vgn$, it intersects the partial toroidal boundary, and if the intersection involves $\vgn$ dimension in one less, then that also intersects the partial toroidal boundary, etc. 

Since the diagonal $\calD_g\subset\calA_g$ clearly intersects the generic boundary stratum $\calA_{g-1}\subset\partial\calA_g^{\op{Sat}}$ of the Satake compactification, just by sending one of the~$t_i\in\HH_1$ to $i\infty=\partial\calA_1$ constructs such a degeneration, this means that it also intersects the partial toroidal boundary. Moreover its intersection with partial toroidal boundary will again involve the (preimage in the universal family) of the diagonal $\calD_{g-1}$, which will again intersect the partial toroidal boundary, and thus the inductive proof of~\cite[Thm.~13]{grsmconjectures} applies to any irreducible component of $\vgn$ containing $\calD_g$. The local smoothness statement of the theorem above is new.

We note that it appears much harder to try to apply such degeneration arguments for the study of the Hessian rank loci in \Cref{thm:tn2} and \Cref{thm:tn3}. While the degeneration of the derivatives of theta constants to the boundary of the partial compactification is well-known, the Hessian matrix would involve second order derivatives of the theta constants in genus~$g-1$, but also first order derivatives of theta functions in genus~$g-1$ evaluated at the point of the abelian variety that gives the semiabelian extension data, and thus the rank condition appears much harder to work with by induction in genus.
\end{rem}

\smallskip
In~\Cref{sec:notat} we recall the basic notions about theta functions and the action of the symplectic group on the set of theta characteristics. In~\Cref{sec:characterize} we recall the classical characterizations of the hyperelliptic and decomposable loci $\calHJ_g$ and $\calR_g$ within $\calA_g$ in terms of vanishing of theta constants. In~\Cref{sec:expand} we set up convenient notation for writing down the expansions of the theta function and its derivatives near the locus $\II_g$ of diagonal period matrices. One new technical result that we prove is the description of the action of the stabilizer group of the theta function with characteristic $m$ on the set of irreducible components of $\DD_g$. In~\Cref{sec:hyp} we reprove Shepherd-Barron's result on the infinitesimal structure of $\calHJ_g$ near $\calD_g$. In~\Cref{sec:gn} we prove~\Cref{thm:Gg} on the locus~$\vgn$. Finally, in~\Cref{sec:tn} we prove Theorems~\ref{thm:tn2} and~\ref{thm:tn3} on the Hessian rank loci $\vtn^2$ and $\vtn^3$.

\subsection*{Acknowledgments}
We are grateful to Daniele Agostini, Lynn Chua, and Nick Shepherd-Barron for sharing with us their preprints \cite{agch} and \cite{sbpoincare}, respectively, and their interesting ideas, and thus reigniting our investigation of this subject. We are indebted to Hershel Farkas, a collaboration with whom on~\cite{fagrsm} led us to investigate and appreciate the importance of expansions of theta functions near the diagonal. The second author is grateful to Enrico Arbarello and Edoardo Sernesi for valuable conversations and correspondence in the past years.

\section{Notation: theta functions and level covers}\label{sec:notat}
We denote by $\HH_g:=\{\tau\in\op{Mat}_{g\times g}(\CC)\mid \tau=\tau^t, \op{Im}\tau>0\}$ the Siegel upper half-space of complex symmetric matrices with positive definitive imaginary part. It is a homogeneous space for the action of $\Sp(2g,\RR)$, where an element
$$
 \sigma=\left(\begin{matrix} A & B \\ C & D\end{matrix}\right) \in Sp(2g,\RR)
$$
acts via
$$
\sigma\cdot \tau:=(A\tau+B)(C\tau+D)^{-1}\,.
$$
We denote by $\Gamma_g:=\Sp(2g,\ZZ)$ the Siegel modular group, and let $\Gamma_g(n):=\{\sigma\in\Gamma_g:\sigma\equiv 1_{2g}\mod n\}$ (where from now on we denote by $1_k$ the $k\times k$ identity matrix) denote the principal congruence subgroup of $\Gamma_g$. The quotient $\calA_g=\HH_g/\Gamma_g$ is the moduli space of complex principally polarized abelian varieties (ppav), and $\calA_g(n)=\HH_g/\Gamma_g(n)$ is the moduli space of ppav with a choice of a full symplectic level~$n$ structure. Recall that $\calJ_g$ and $\calHJ_g$ denote the closures in~$\calA_g$ of the loci of Jacobians and of hyperelliptic Jacobians, respectively.

We denote by $p:\HH_g\to\calA_g$ and $p_n:\HH_g\to\calA_g(n)$ the quotient maps, and by abuse of notation will also denote the same way their restrictions to various submanifolds such as $\JJ_g:=p^{-1}(\calJ_g)$ or $\HJ_g:=p^{-1}(\calHJ_g)$. For a subvariety $\calX\subset\calA_g$, we will also write $\calX(n)$ to denote its preimage on a level cover: $\calX(n):=p_n(p^{-1}(\calX))\subset\calA_g(n)$. Very importantly, we note that since $p:\HH_g\to\calA_g$ is a Galois cover, for any irreducible subvariety $\calX\subset\calA_g$, for any two irreducible components $\XX'$ and $\XX''$ of $p^{-1}(\calX)\subset\HH_g$, there must exist an element $\gamma\in\Gamma_g$ mapping $\XX'$ to $\XX''$.

We call a ppav decomposable if it is isomorphic to a product of two lower-dimensional ppav. Analytically, $\tau$ is decomposable if and only if there exists $\sigma\in\Gamma_g$, such that
$$
 \sigma\cdot\tau=\left(\begin{smallmatrix} \tau_1 & 0 \\
0& \tau_2
\\
\end{smallmatrix}\right), \quad {\rm with }\, \tau_i\in \HH_{g_i},\,\, g=g_1+g_2,\,\, g_1,g_2>0\,. $$
(classically, such ppav are called reducible). We denote by 
$$
 \calR_g:=(\calA_1\times\calA_{g-1})\cup(\calA_2\times\calA_{g-2})\cup\dots\subset\calA_g
$$ 
the locus of decomposable ppav, and denote $\RR_g:=p^{-1}(\calR_g)\subset\HH_g$ its preimage in the Siegel space. We recall that $\calD_g=\calA_1\times\dots\times\calA_1\subset\calA_g$ denotes the locus of products of elliptic curves, and $\DD_g:=p^{-1}(\calD_g)\subset\HH_g$ denotes its preimage, of which $\II_g$ is an irreducible component. We thus have $\calD_g\subset\calHJ_g\subset\calJ_g\subset\calA_g$ and $\calD_g\subset\calR_g\subset\calA_g$.

The goal of this paper is to describe these loci locally near $\calD_g$, and the main tool will be by analyzing the Taylor expansions of theta functions near~$\II_g$. Recall that the theta function with characteristics $\ep,\de\in\ZZ_2^g$ is the function of $\tau\in\HH_g$ and $z\in\CC^g$ given by
$$
 \tc\ep\de(\tau,z):=\sum\limits_{p\in\ZZ^g} \exp \pi i\left[^t(p+\tfrac{\ep}{2})\tau(p+\tfrac{\ep}{2})+2 ^t(p+\tfrac{\ep}{2})(z+\tfrac{\de}{2})\right]\,.
$$
We will write theta characteristics also as $m=\chars\ep\de\in\ZZ_2^{2g}$; we will usually write $\ep,\de$ as rows (or sometimes columns, if notationally more convenient) of $g$ zeroes and ones, and operate with them over $\ZZ_2$ unless stated otherwise. In particular,~$m$ is called even or odd depending on whether the scalar product $\ep\cdot\de$ is zero or one as an element of $\ZZ_2$. As a function of~$z$, the theta function is even or odd, respectively. The theta constants are the values of theta functions at $z=0$, and theta gradients are the values of the $z$-gradient of the theta function, evaluated at $z=0$. We will drop the $z$ variable from notation in both cases, and write
$$
 \theta_m(\tau):=\theta_m(\tau,0)\in\CC;\qquad \grad\theta_m(\tau):=\left\lbrace\tfrac{\partial}{\partial z_a}\theta_m(\tau,z)|_{z=0}\right\rbrace_{a=1,\ldots,g}\in\CC^g\,.
$$
Note that theta constants vanish identically for $m$ odd, while theta gradients vanish identically for $m$ even. Theta constants and theta gradients are examples of scalar- (resp.~vector-) valued Siegel modular forms, i.e.~are sections of a suitable line (resp.~rank $g$ vector) bundle on a suitable cover of $\calA_g$. In fact these are modular forms with non-trivial multiplier with respect to $\Gamma_g(2)$. Moreover, $\Gamma_g$ acts on theta characteristics, considered as elements of $\ZZ_2^{2g}$, via an affine-linear action of its quotient $\Sp(2g,\ZZ_2)=\Gamma_g/\Gamma_g(2)$. This action is given explicitly by
\begin{equation}\label{eq:charaction}
 \sigma\circ \chars{\ep}{\de}:=\left(\begin{smallmatrix}D&-B \\ -C&A\end{smallmatrix}\right) \chars{\ep}{\de}+ \chars{{\rm
 diag}(c \,^t d)} {{\rm diag}(a\,^t b)}\,.
\end{equation}
We refer to \cite{igusabook} for further details, and note that $\Gamma_g(2)$ is precisely the subgroup of~$\Gamma_g$ that fixes every characteristic.

We recall from~\cite{igusabook,smlevel2} that the orbits of $\Gamma_g$ on tuples of characteristics are fully characterized by parity of characteristics, by the a/syzygy properties of triples of characteristics, and by linear relations with an even number of terms.

We will not use the details of this except to note that the zero loci of $\theta_m(\tau)$ and $\grad\theta_m(\tau)$, 
$$
\begin{aligned}
 \tn\chars\ep\de&:=\{\tc\ep\de(\tau)=0\}\subset\HH_g\quad{\rm and}\\
\gn\chars\ep\de&:=\{\grad\tc\ep\de(\tau)=0\}\subset\HH_g\,,
\end{aligned}
$$
are invariant under the action of~$\Gamma_g(2)$, and thus are preimages of well-defined subvarieties in $\calA_g(2)$.

For any $2\le k\le g$ we define $\tn^k\chars\ep\de\subset\tn\chars\ep\de$ to be the locus where the rank of the Hessian matrix $\left(\partial_{z_a}\partial_{z_b}\tc\ep\de(\tau,z)|_{z=0}\right)_{1\le a,b\le g}$ is at most $k$; by abuse of notation, we will use this notation for both a subvariety of $\calA_g(2)$ and an analytic subset of $\HH_g$, when no confusion can arise.

Since the action of $\Gamma_g/\Gamma_g(2)$ permutes theta characteristics and the loci $\tn\chars\ep\de$ and respectively $\gn\chars\ep\de$ transitively, it follows that their images
$$
 \vtn:=p(\tn\chars\ep\de)\subset\calA_g\quad{\rm and}\quad \vgn:=p(\gn\chars\ep\de)\subset\calA_g
$$
are independent of the choices of even or odd characteristic~$\chars\ep\de$, respectively. Geometrically,~$\vtn$ is the locus of ppav whose theta divisor has a singularity (necessarily of even multiplicity, at least 2) at an even two-torsion point, while~$\vgn$ is the locus of ppav whose theta divisor has a singularity (necessarily of odd multiplicity, at least 3) at an odd two-torsion point of the abelian variety.
The loci $\vtn^k:=p(\tn^k\chars\ep\de)\subset\calA_g$ are similarly independent of the choice of characteristic~$\chars\ep\de$.

For low genera these loci have a simple geometric interpretation, which is part of the motivation for studying them:
$$
\begin{array}{clll}
 g=2:&&\vtn=\calR_2=\calA_1\times\calA_1& \vgn=\emptyset \\
 g=3:&\vtn^2=\calR_3&\vtn=\calHJ_3&\vgn=\calA_1\times\calA_1\times\calA_1\\
 g=4:& \vtn^2=\calR_4&\vtn^3=\calJ_4\cap\vtn &\vgn=\calA_1\times\calHJ_3\\
 g=5:&&&\vgn=(\calA_1\times\vtn)\cup \calI\calJ\,,
\end{array}
$$
where in genera 4 and 5 the locus $\vtn$ does not admit such a quick geometric description (while genus 4 curves with a theta-null are canonical curves that lie on a singular quadric, there is no similarly easy description for principally polarized abelian fourfolds with a theta-null), and $\calI\calJ$ denotes the closure of the locus of intermediate Jacobians of cubic threefolds; see~\cite{grhu1} for more discussion of the cases $g=4,5$.

The locus~$\vtn$ was studied classically, and the first result in this study is that~$\vtn$ is always an irreducible divisor in~$\calA_g$ \cite[p.~88]{freitagbooksingular}, while in~\cite{grsmordertwo} we conjectured that~$\vgn$ is always of pure codimension~$g$ in~$\calA_g$, and proved this for every irreducible component of~$\vgn$ that intersects the boundary of the partial compactification of $\calA_g$. 

One can easily see that $\calR_g\subset\vtn^2$, while Riemann theta singularity theorem for Jacobians implies the inclusion $\calJ_g\cap\vtn\subset\vtn^3$. In~\cite{grsmgen4} we proved the conjecture of H.~Farkas that in genus 4 the equality $\calJ_4^\circ\cap\vtn=\vtn^3\setminus\vtn^2$ holds. One of our main results is~\Cref{thm:tn3}, extending this genus 4 statement, and the recent genus 5 result of Agostini and Chua to arbitrary genus, showing that $\calJ_g^\circ\cap\vtn$ is an irreducible component of~$\tn^3$ for any genus.

\smallskip
One technical point that pervades our work is whether we work on $\calA_g$, $\HH_g$, or (as the above discussion shows is often useful) on $\calA_g(2)$. Note for example that while $\vtn\subset\calA_g$ is irreducible, $\vtn(2)=p_2(p^{-1}(\vtn))=\cup\vtn\chars\ep\de\subset\calA_g(2)$ has $2^{g-1}(2^g+1)$ irreducible components, indexed by characteristics. However, by \cite{mosm} for any $g\geq 3$ the analytic spaces $\tn\chars\ep\de\subset\HH_g$ are irreducible for each $\chars\ep\de$.

\section{The decomposable and hyperelliptic loci}\label{sec:characterize}
In this section we recall the known characterizations of~$\calR_g$ and~$\calHJ_g$ in terms of vanishing of certain sets of theta constants, and study the combinatorics of the relevant characteristics. This is equivalent to describing the irreducible components of the corresponding loci on the level covers. For further use we denote
\begin{equation}\label{eq:scalar}
\langle\ep,\de\rangle:=\sum_{a=1}^g \ep_a\cdot\de_a\in\ZZ
\end{equation}
the scalar product of $\ep$ and $\de$, considered in~$\ZZ$ (unlike the usual pairing $\ep\cdot\de\in\ZZ_2$). We further denote by $\calE$ the set of all even characteristics, and for any even $\ell\in\ZZ_{\ge 0}$ denote by $\calE_\ell\subset\calE$ the set of all even characteristics such that $\langle\ep,\de\rangle=\ell$, and denote $\calE^*:=\calE\setminus\calE_0$. We will similarly decompose the set~$\calO$ of odd characteristics as $\calO=\sqcup_{1\le \ell \le g,\ \ell\textrm{ odd }}\calO_\ell$, and denote $\calO^*:=\calO\setminus\calO_1$ to exclude the ``simplest'' odd characteristics.

\subsection{The decomposable locus, and the diagonal}
Recall that the theta function near a block-diagonal period matrix factorizes as follows:
\begin{equation}\label{eq:factorize}
\tc{\ep_1\ep_2}{\de_1\de_2}\left(\left(\begin{smallmatrix} \tau_1&0\\ 0&\tau_2\end{smallmatrix}\right),\,\begin{smallmatrix} z_1\\ z_2\end{smallmatrix}\right)=\tc{\ep_1}{\de_1}(\tau_1,z_1)\cdot\tc{\ep_2}{\de_2}(\tau_2,z_2)
\end{equation}
for any $\tau_i\in\HH_{g_i}$ and $z_i\in\CC^{g_i}$ with $g_1+g_2=g$. By applying this formula recursively, we see that for a diagonal period matrix $\tau_0=\left(\begin{smallmatrix}t_1&0&\dots&0\\ 0&t_2&\dots&0\\ \vdots&\vdots&\ddots&\vdots\\ 0&0&\dots&t_g\end{smallmatrix}\right)\in\II_g$ the value of the theta constant is given by
$$
 \tc\ep\de(\tau_0)=\tc{\ep_1}{\de_1}(t_1)\cdot\ldots\cdot\tc{\ep_g}{\de_g}(t_g)\,.
$$
We recall the characterization of block-diagonal period matrices:
\begin{prop}[{\cite[Theorem 5]{smlevel2}}]\label{prop:vanishproducts} 
A period matrix $\tau\in\HH_g$ lies in the $\Gamma_g(2)$ orbit of the locus $\HH_{g_1}\times\HH_{g_2}$ if and only if $\tc{\ep_1&\de_1}{\ep_2&\de_2}(\tau)=0$ for all pairs of odd characteristics $\chars{\ep_1}{\de_1}\in\ZZ_2^{2g_1}$ and $\chars{\ep_2}{\de_2}\in\ZZ_2^{2g_2}$.
\end{prop}
By applying this proposition and using the factorization formula~\eqref{eq:factorize} recursively, one obtains
\begin{cor}\label{cor:diagonal}
A period matrix $\tau\in\HH_g$ lies in the $\Gamma_g(2)$ orbit of the locus $\II_g$ if and only if all theta constants $\tc\ep\de$ such that at least one column $\chars{\ep_a}{\de_a}$ of $\chars\ep\de$ is equal to~$\chars11$ vanish at $\tau$. 
\end{cor}
In our notation, this corollary can be reformulated as the statement that $\tau\in\Gamma_g(2)\circ\II_g$ if and only if $\tau\in\tnm$ for all $m\in\calE^*$.

\Cref{prop:vanishproducts} and~\Cref{cor:diagonal} characterize the loci of block-diagonal and diagonal period matrices in~$\HH_g$, and their images in~$\calA_g(2)$; to obtain from these a characterization of the loci~$\calR_g$ and $\calD_g$ in~$\calA_g$ one needs to consider the orbits of the action of $\Gamma_g$ on the locus of block-diagonal period matrices. The (setwise) stabilizers of such loci are known classically \cite{freitagF}:
\begin{prop}\label{prop:stabproduct}
The (setwise) stabilizer $\Stab_{\Gamma_g}(\HH_{g_1}\times\HH_{g_2})$ is equal to the direct product $Stab_{g_1,g_2}:=\Gamma_{g_1}\times \Gamma_{g_2}$, except for the case $g_1=g_2=g/2$, when the stabilizer is the semi-direct product $Stab_{g/2,g/2}:=(\Gamma_{g/2}\times \Gamma_{g/2})\ltimes S_2$ with the involution interchanging the two blocks.
\end{prop}

\begin{cor}\label{prop:stabdiagonal}
The (setwise) stabilizer of the diagonal $\Stab_{\II_g} \subset \Gamma_g $ is equal to the wreath product $Stab_{\II_g}:=\Gamma_1\wr S_g$, i.e.~is the semidirect product of $(\Gamma_1)^{\times g}=\SL(2,\ZZ)^{\times g}$ and the permutation group $S_g$ of $g$ elements.
\end{cor}
Here we think of the permutation group $S_g$ as embedded into~$\Gamma_g$ as block matrices with two off-diagonal $g\times g$ blocks equal to zero, and two on-diagonal $g\times g$ blocks equal to each other, and each being a permutation matrix.
 
For our purposes, we are interested in finding an explicit manageable subgroup acting transitively on the sets $\calE$ and $\calO$, and we will find such a subgroup that contains $\Stab_{\II_g}$, but is slightly larger. This is a manifestation of the idea we'll use later in the paper: instead of just considering diagonal period matrices, we will allow some $2\times 2$ blocks, and will enlarge the group correspondingly.

We thus set $G_g:=\Gamma_2\times (\Gamma_1)^{\times (g-2)} \wr S_g$, where we think of the first factor as block-diagonal period matrices with one $2\times 2$ block and $g-2$ blocks of size $1\times 1$, and $S_g$ is embedded into $\Sp(2g,\ZZ)$ as before.
\begin{lm}\label{lm:orbita}
The group $G_g$ acts transitively on each of the two sets of characteristics $\calE$ and $\calO$.
\end{lm}
\begin{proof}
We do the even case, the odd case being completely analogous.

We first permute the coordinates so that all the $\ell$ columns of characteristic that are equal to $\chars11$ appear first. Then since $\Gamma_1$ acts transitively on the set of 3 even characteristics in genus one, acting by $\Gamma_1$ on the characteristics in each even column maps them all to $\chars00$. Altogether, this shows that for a fixed $\ell$, the stabilizer $\Stab_{\II_g}\subset\Gamma_g$ acts transitively on the set $\calE_\ell$. Since $\Stab_{\II_g}\subset G_g$, it is thus enough to show that $G_g$ can change $\ell$ arbitrarily. For this, we observe that the element 
$$\sigma_0: =\left(\begin{smallmatrix} 1 & 1 &1 & 0 \\ 1 & 1 & 0 & 1\\ 0&1&1 & 0 \\1&0&0&1 \end{smallmatrix}\right)\in\Gamma_2$$
sends $\chars{00}{00}$ to $\chars{11}{11}$.
Once the columns of a characteristic are permuted so that the first $\ell$ ones are equal to $\chars11$, we apply $\sigma_0$ in the first two coordinates to make the first two columns equal to $\chars00$, thus going from a characteristic in $\calE_\ell$ to a characteristic in $\calE_{\ell-2}$. Repeating this process shows that the $G_g$ orbit of any characteristic in $\calE_\ell$ contains a characteristic in $\calE_0$.
\end{proof}

\subsection{The hyperelliptic locus}\label{sec:hypcomp}
We recall from ~\cite{tsu} and \cite{mosm} that the irreducible components of $\HJ_g\subset \HH_g$ are in bijection with (and are in fact preimages of) the irreducible components of $\calHJ_g(2)\subset\calA_g(2)$. We will describe one such component explicitly, which will suffice since $\Gamma_g$ acts transitively on the set of irreducible components of $\HJ_g$ or $\calHJ_g(2)$.

We say that a set $m_0,\ldots,m_{2g}$ of characteristics is called an essential basis if any characteristic $m\in\ZZ_2^{2g}$ can be written uniquely as a sum of an odd number of $m_i$'s. It follows from the description of the action that an element of $\Gamma_g$ lies in $\Gamma_g(2)$ (equivalently, fixes all characteristics) if and only if it fixes every element of a chosen essential basis.

Recall that a special fundamental system of characteristics is a set of~$g$ odd characteristics and~$g+2$ even characteristics such that every triple of characteristics is azygetic. The description of the orbits of the action of~$\Gamma_g$ on tuples of characteristics implies that $\Gamma_g$ acts transitively on the set of special fundamental systems. Moreover, the condition of being azygetic implies that any subsequence of a fundamental system is a sequence of essentially independent characteristics,i.e. the sum of an  even  number of characteristics is  always  different from $0$,  see~\cite{igusajacobi}. As a consequence, eliminating any characteristic from any special fundamental system of characteristics gives an essential basis.
We now fix the following special fundamental system:
\begin{equation}\label{eq:specsystem}
I:=\left(o_1 ,\dots, o_{g}, e_1,\dots, e_{g+2}\right)=
\left(\begin{smallmatrix}
1&0&0&\dots&0&0&0&1&0&0&\ldots&0&0&0\\
0&1&0&\dots&0&0&0&0&1&0&\ldots&0&0&0\\
0&0&1&\dots&0&0&0&0&0&1&\dots&0&0&0\\
\vdots&\vdots&\vdots&\ddots&\vdots&\vdots&\vdots&\vdots&\vdots&\vdots&\ddots&\vdots&\vdots&\vdots\\
0&0&0&\dots&1&0&0&0&0&0&\dots&1&0&0\\
0&0&0&\dots&0&1&0&0&0&0&\dots&0&1&0\\[6pt]
\hline\\[6pt]
1&1&1&\dots&1&1 &0&0& 1&1&1&\dots&1&1\\
0&1&1&\dots&1&1 &0&0& 0&1&1&\dots&1&1\\
0&0&1&\dots&1&1 &0&0& 0&0&1&\dots&1&1\\
\vdots&\vdots&\vdots&\ddots&\vdots&\vdots& \vdots&\vdots& \vdots&\vdots&\vdots&\ddots&\vdots&\vdots\\
0&0&0&\dots&1&1 &0&0& 0&0&0&\dots&1&1\\
0&0&0&\dots&0&1 &0&0& 0&0&0&\dots&0&1
\end{smallmatrix}\right)\,,
\end{equation}
where we have denoted the $g$ odd characteristics by~$o_j$, and the~$g+2$ even characteristics by $e_j$. We further denote
\begin{equation}\label{eq:bdefined}
b^g:=o_1+\dots+o_g=\chars{1&1&\dots&1&1}{\tfrac{1-(-1)^g}{2}&\tfrac{1-(-1)^{g-1}}{2}&\dots&0&1}\equiv
\chars{1&1&\dots&1&1}{g&g-1&\dots&2&1}\mod 2
\end{equation}
the sum of the odd characteristics in this special fundamental system.

If we exclude $e_1=\chars{0\dots0}{0\dots 0}$, the remaining $2g+1$ characteristics of the special fundamental system~$I$ form an essential basis. Thus any characteristic $m\in\ZZ_2^{2g}$ can be written as a sum of an odd number among these $2g+1$ characteristics. Since the sum of all characteristics in~$I$ is zero, the sum of any subset of characteristics in~$I$ is equal to the sum of the complementary subset of characteristics in~$I$. Thus altogether we see that every characteristic~$m$ can be written uniquely as a sum of at most $g$ among the characteristics $o_1,\dots,o_g,e_2,\dots,e_{g+2}$. Suppose~$m$ is the sum of~$k$ among these characteristics; then it can be checked that the characteristic $m+b^g$ is even if and only if $k\equiv g$ or $k\equiv g+1\mod 4$. We then have the following
\begin{prop} [\cite{mumfordbooktheta2}, \cite{poor}, \cite{tsu}]\label{prop:hyponecomponent}
There exists an irreducible component $\HJ_g^{I}$ of $\HJ_g$ that is defined by the equations
\begin{equation}\label{eq:hypvanish}
\theta_{m+b^g}(\tau)=0
\end{equation}
for all~$m$ that are equal to a sum of {\em strictly less than} $g$ among the characteristics $o_1,\ldots,o_g,e_2,\ldots,e_{g+2}$.
\end{prop}
\begin{rem}
The actual result of Mumford is that if $\tau\in\HH_g$ is such that $\theta_{m+b^g}(\tau)\ne 0$ if {\em and only if} $k=g$, then $\tau\in\HJ_g^\circ$. In the above proposition we do not require the non-vanishing of those $\theta_{m+b^g}(\tau)$ where~$m$ is the sum of precisely~$g$ elements of the special fundamental system. Thus clearly the locus described in the proposition contains an irreducible component of~$\HJ_g$. Furthermore, Poor~\cite{poor} proved that the vanishing conditions~\eqref{eq:hypvanish} by themselves cut out an irreducible component of $\HJ_g\setminus\RR_g$.
\end{rem}
The action of~$\Gamma_g$ is transitive on the set of all special fundamental systems,
and thus one has the following characterization of the hyperelliptic locus:
\begin{prop}[\cite{mumfordbooktheta2}, \cite{poor}]\label{prop:allhyp}
An indecomposable period matrix $\tau\in\HH_g\setminus\RR_g$ lies in $\HJ_g$ if and only if there exists a special fundamental system $o'_1,\ldots,o'_g$, $e'_1,\dots,e'_{g+2}$ such that defining $b'^g:=o'_1+\dots+o'_g$, the theta constant $\theta_{m+b'^g}$ vanishes at $\tau$ if and only if $m$ can be written as a sum of strictly less than $g$ elements of the special fundamental system.
\end{prop}
\begin{rem}\label{rem:hpo}About the odd counterpart, we observe that, when $g\geq 5$, the theta gradient with characteristic $m_{I}=o_1+o_2+o_3+o_4+o_5=\chars {111110\dots0}{101010\dots0}$ vanishes along the component $\HJ_g^{I}$. Indeed, $$m_{I}=b^g+o_6+\dots+o_g$$ is the sum of $b^g$ and~$g-5$ elements of the special fundamental system, and this condition implies the vanishing of the gradient of the theta function at the hyperelliptic point $\tau$, see~\cite{igusajacobi}.
\end{rem}
\begin{rem}
We will show that \Cref{thm:hypsmall} applies to $\HJ_g^{I}$ , though in fact the irreducible component that Shepherd-Barron uses in~\cite{sbpoincare} is a different one. Note that $\II_g$ is contained in multiple irreducible components of $\HJ_g$, see the discussion after the proof of~\Cref{thm:hypfull}.
\end{rem}
 In  \cite{tsu}, Tsuyumine studies the intersection of irreducible components of $\calR_g(2)$ and of $\calHJ_g(2)$. He also shows that the stabilizer of the component $\calHJ_g^I$ of $\calHJ_g(2)$ is isomorphic to the symmetric group $S_{2g+2}$. Moreover he considers also the boundary components of $\calHJ_g^I$ contained in $\calR_g (2)$. While his analysis again is only for decomposable ppav that are products of two indecomposable ones, his analysis extends in full generality to yield the statement that all boundary components related to a decomposition $ g= g_1+\dots+g_k$ are conjugated via the stabilizer subgroup at $\calHJ_g^I$:
\begin{lm}\label{lm:tsu}
For any two irreducible components $Z$ and $W$ of $\DD_g$ contained in an irreducible component $X$ of $\HJ_g$, there exists $\sigma \in Stab_{X}$ such that $\sigma( Z)= W$. 

For any two irreducible components $X$ and $Y$ of $\HJ_g$ containing $\II_g$, there exists $\sigma \in Stab_{\II_g}$ such that $\sigma( Y)= X$.
\end{lm}
\begin{proof} 
For the first statement, recall that as already discussed in \Cref{sec:hypcomp}, the irreducible components of $\HJ_g$ are in bijection with those of $\calHJ_g(2)$. Hence the stabilizer of $X$ acts transitively on the set of all its boundary components related to a decomposition $ g= 1+\dots+1$. 

For the second statement, since the cover $\HH_g\to\calA_g$ is Galois, being the quotient by $\Gamma_g$, we know that $\Gamma_g$ acts transitively on the set of irreducible components of $\HJ_g$, and thus there exists some $\sigma_1\in\Gamma_g$ such that $X=\sigma_1(Y)$. Denoting $\II'_g:=\sigma_1(\II_g)$ the irreducible component of $\DD_g$ that $\II_g$ is mapped to, by the first statement there exists $\sigma_2\in Stab_X$ such that $\sigma_2(\II'_g)=\II_g$. Thus $\sigma:=\sigma_2\circ \sigma_1$ satisfies $\sigma(\II_g)=\sigma_2(\II'_g)=\II_g$ and maps $Y$ to $X$, as required.
\end{proof}

\section{Expansions of theta functions near the diagonal}\label{sec:expand}
Our main computational tool is working with Taylor expansions of defining equations of our loci near $\II_g$. We will work in a sufficiently small analytic neighborhood~$U$ of~$\II_g$; that is, we will fix arbitrary generic $t_1,\dots,t_g\in \HH_1$, and assume that all $\tau_{ab}$ with $a<b$ satisfy $|\tau_{ab}|<\ep$ for some sufficiently small~$\ep$ (small compared to all $t_i$). Since the diagonal period matrix $\diag(t_1,\dots,t_g)$ lies in the open set~$\HH_g$, so doing this for every $t_1,\dots,t_g$ we get an open neighborhood $\II_g\subset U\subset\HH_g$. We can thus expand theta constants , theta gradients, etc. with respect to all the variables~$\tau_{ab}$ for $1\le a<b\le g$ at a fixed generic point $\diag(t_1,\dots,t_g)\in\II_g$. 

The Taylor expansion of theta constants near~$\II_g$ was recently used in our work~\cite{fagrsm} with H.~Farkas on the Schottky problem, and we now recall it. We also give the formulas for the Taylor expansions of theta gradients near~$\II_g$, and for the Hessian of the theta function. These are the formulas that will make all of our results work, and we introduce various conventions to be able to keep track of the formulas in a reasonable way.

First of all, we recall that by~\eqref{eq:factorize} the theta constant near a diagonal period matrix in~$\II_g$ factorizes. Furthermore, the $z$-derivatives of the theta constant factorize the same way. Thus recalling the heat equation satisfied by theta functions
$$
\frac{\partial\tc\ep\de(\tau,z)}{\partial_{\tau_{jj}}}=\frac{1}{4\pi i}\frac{\partial^2\tc\ep\de(\tau,z)}{\partial z_j\partial z_j};\quad \frac{\partial\tc\ep\de(\tau,z)}{\partial_{\tau_{jk}}}=\frac{1}{2\pi i}\frac{\partial^2\tc\ep\de(\tau,z)}{\partial z_j\partial z_k}
\ \textrm {for\ } j\ne k
$$
allows us to evaluate at a point of~$\II_g$ the derivatives of $\tc\ep\de$ with respect to $\tau$. We will be expanding theta functions and their derivatives in Taylor series near~$\II_g$ with respect to the variables $\tau_{ab}$ for all $1\le a<b\le g$, with each term of the expansion being a function in the variables $\tau_{11},\dots,\tau_{gg}$, which we will denote $t_1,\dots,t_g$, to remember that they are elements of the upper half-plane. To shorten the formulas in the rest of the text, we adopt the following
\begin{conv}\label{conv1}
For $t_j\in\HH_1$ and fixed $\chars\ep\de$ we denote by
$$\vt_j:=\tc{\ep_j}{\de_j}(t_j);\quad \vt_j':=\frac{\partial}{\partial z}\tc{\ep_j}{\de_j}(t_j,z)|_{z=0};\quad \vt_j'':=\dots$$
the one-variable theta functions and their $z$-derivatives evaluated at $z=0\in\CC$ (which may be identically zero depending on the parity of $\chars{\ep_j}{\de_j}$).

We write $O(\ep^N)$ to signify a sum of monomials of total degree at least $N$ in all the variables $\tau_{ab}$, for all $1\le a<b\le g$.

Even with this notation, the formulas, as in~\cite{fagrsm}, would get very complicated, so we introduce further conventions to make them more readable. For a characteristic of the form $\chars\ep\de=\chars{1\dots10\dots0}{1\dots10\dots0}$ with the first $\ell$ columns equal to $\chars11$, we will use capital letters $J,K,\dots$ to denote columns where the characteristic is $\chars11$, i.e.~$1\le J\le \ell$, and we will use small letters $j,k,\dots$ to denote columns where the characteristic is $\chars00$, i.e.~$\ell+1\le j\le g$.

We denote $S_{2n}$ the permutation group on~$2n$ elements, and by $T_{2n}\subset S_{2n}$ the set of permutations that can be written as products of $n$ disjoint transpositions. For $\sigma\in T_{2n}$, we denote by $\mu\subset\sigma$ one of the transpositions $\mu:\alpha_\mu\longleftrightarrow \beta_\mu$ whose product is $\sigma$.

Finally, for a set of an even number of possibly repeating indices $a_1,\dots,a_{2n}\in\{1,\dots,g\}$ we will denote
$$
 [a_1,\dots,a_{2n}]:=\frac{1}{(2\pi i)^n} \sum_{\sigma\in T_{2n}}n_\sigma\prod_{\mu\subset\sigma}\tau_{a_{\alpha_\mu}\, a_{\beta_\mu}}\,,
$$
where the combinatorial coefficient $n_\sigma=a_\sigma\cdot b_\sigma\cdot c_\sigma$ is computed as follows. The factor $a_\sigma$ is equal to $0$ if there exists any $\mu\subset\sigma$ such that $a_{\alpha_\mu}=a_{\beta_\mu}$, and is equal to $1$ otherwise. If there are precisely $N$ elements $\sigma=\sigma_1,\sigma_2,\dots,\sigma_N\in T_{2n}$ such that the resulting monomials are equal, then $b_\sigma$ is set to be equal to $1/N$ --- so that the result is essentially that each distinct monomial is counted exactly once in the sum. Finally, for a given $\mu$ we rewrite
$$\prod_{\mu\subset\sigma} \tau_{a_{\alpha_\mu}\, a_{\beta_\mu}}=\prod_{1\le i\le j\le g}\tau_{ij}^{d_{ij}}$$
by gathering the powers of the same~$\tau_{ij}$ together, and let then $c_\sigma:=1/\prod (d_{ij}!)$. The reason for this last factor is that this is the coefficient with which
$$\prod_{\mu\subset\sigma}\partial_{ \tau_{a_{\alpha_\mu}\, a_{\beta_\mu}}}=\prod_{1\le i\le j\le g}\partial_{\tau_{ij}}^{d_{ij}}$$
appears in the Taylor expansion. Note that $[a_1,\dots,a_{2n}]=O(\ep^n)$, as each summand is a degree~$n$ monomial in the~$\tau$'s.
\end{conv}
To unravel this very useful notation, we give some examples:
$$
 [1,1]=0;\quad[1,2]=\frac{1}{2\pi i}\tau_{12};\quad [1,1,2,3]=\frac{1}{(2\pi i)^2}\tau_{12}\cdot\tau_{13};
$$
$$
 [1,2,3,4]=\frac{1}{(2\pi i)^2}\left(\tau_{12}\cdot\tau_{34}+\tau_{13}\cdot\tau_{24}+\tau_{14}\cdot\tau_{23}\right),
$$
$$
 [1,1,2,2,3,4]=\frac{1}{(2\pi i)^3}\left(\frac12\tau_{12}^2\tau_{34}+\tau_{12}\tau_{13}\tau_{24}+\tau_{12}\tau_{14}\tau_{23}\right)\,.
$$
The reason this notation is so useful for us is that in computing the terms of the expansion of theta functions and their derivatives we use the heat equation repeatedly. Each factor $\tau_{ab}$ arises when the corresponding derivative $\partial_{\tau_{ab}}$ is taken, so then $\vt_{a}$ and $\vt_{b}$ are differentiated, by the heat equation. Each summand of the Taylor expansion of the theta constant in variables $\tau_{ab}$ for all $1\le a<b\le g$ is thus of the form $\prod_{\alpha=1\dots g} \left(\partial^{n_\alpha}\vt_\alpha\right)$ times the following polynomial in $\tau$'s:
$$
[ \underbrace{1,\dots,1}_{n_1},\underbrace{2,\dots,2}_{n_2},\dots,\underbrace{g,\dots,g}_{n_g}]
$$
in our new notation. For the expansion of the derivative $\partial_{\tau_{ab}}\theta_m$, the polynomial multiplying $\prod_{\alpha=1\dots g} \partial^{n_\alpha}\vt_\alpha$ is similar, except that $a$ and $b$ will be included in the expression $n_a-1$ and $n_b-1$ times, respectively.

As a warm-up, we write down in genus~$4$ the expansion of $\theta_m=\tc{1100}{1100}$ near $\II_4$, using this notation.
\begin{equation}
\begin{aligned}
&\theta_m=\theta_m(\diag(t_1,t_2,t_3,t_4))+\sum_{1\le a<b\le g}\tau_{ab}\frac{\partial\theta_m}{\partial \tau_{ab}}(\diag(t_1,t_2,t_3,t_4))+\dots\\
&=\vt_1'\vt_2'\vt_3\vt_4\\&\cdot\left([1,2]+\frac{\vt_1'''}{\vt_1'}[1,1,1,2]+\frac{\vt_2'''}{\vt_2'}[2,2,1,2]
+\frac{\vt_3''}{\vt_3}[3,3,1,2]+\frac{\vt_4''}{\vt_4}[4,4,1,2]\right)+O(\ep^3)\,.
\end{aligned}
\end{equation}
Of course all of the terms above can easily be written out explicitly, the terms $[1,1,1,2]$ and $[2,2,1,2]$ are in fact zero, but note that our conventions make the formula readable. We note in particular that the lowest order term of this expansion is $[1,2]=O(\ep)$, while in general for $\chars\ep\de\in\calE_\ell$, the lowest order term of the expansion would be of order $O(\ep^{\ell/2})$. From now on, we denote $s:=\ell/2$, for even characteristics.
\begin{conv}\label{conv:phi}

To make the formulas still nicer, we finally denote, for a given $\ell$ with $1\le\ell\le g$ (for thinking about even characteristics in~$\calE_\ell$ or odd characteristics in $\calO_\ell$)
$$
 f_\alpha:=\begin{cases}
 \vt_\alpha', & \mbox{if } 1\le\alpha\le\ell \\
 \vt_\alpha, & \mbox{if } \ell+1\le\alpha\le g.
 \end{cases}
$$
and denote
$$
 \phi_\alpha:=\frac{f_\alpha''}{f_\alpha};\qquad \psi_\alpha:=\frac{f_\alpha''''}{f_\alpha}-\phi_\alpha^2\,,
$$
where we now use the index $1\le \alpha\le g$.
\end{conv}

We are now ready to compute the two lowest order terms of the expansion of the theta constant with characteristics $\chars{1\dots10\dots0}{1\dots10\dots0}$ for arbitrary $g$ and arbitrary $2\le \ell=2s\le g$:
\begin{equation}\label{eq:thetaexpand}
\theta_m=\left([1,\dots,\ell]+\sum_\alpha \phi_\alpha\cdot[\alpha,\alpha,1,\dots,\ell]\right)\cdot\prod_\alpha f_\alpha+O(\ep^{s+2})
\end{equation}
For further use, we denote
\begin{equation}\label{eq:Xell}
X_\ell:=[1,\dots,\ell];\quad Y_\ell:=\sum_\alpha \phi_\alpha\cdot[\alpha,\alpha,1,\dots,\ell]
\end{equation}
these two leading terms (for $\ell$ of any parity), so that the above becomes $\theta_m=X_\ell+Y_\ell+O(\ep^{s+2})$.

\smallskip
\noindent Similarly, we can obtain formulas for the expansions of the derivatives, where recall we use indices $1\le J,K\le \ell$ and $\ell+1\le j,k\le g$. We first deal with the $z$-derivatives, to be used in our investigation of the locus $\vgn$. In this case the expansion of the theta gradient $\grad\tc{1\dots 10\dots 0}{1\dots 10\dots 0}$ where the characteristic lies in $\calO_\ell$, i.e. has an odd number $\ell=2s+1$ of columns equal to $\chars11$, is
\begin{equation}\label{eq:expgrad1}
\frac{\partial\tc\ep\de}{\partial z_J}(\tau)=\left([1,\dots,\widehat J,\dots,\ell]+\sum_\alpha \phi_\alpha\cdot [\alpha,\alpha,1,\dots,\widehat J,\dots,\ell]\right)\prod f_\alpha +O(\ep^{s+2})
\end{equation}
and
\begin{equation}\label{eq:expgrad2}
\begin{aligned}
\frac{\partial\tc\ep\de}{\partial z_j}(\tau)&=\left(\phi_j\cdot[j,1,\dots,\ell]+\psi_j\cdot[j,j,j,1,\dots,\ell]\right.\\ &\left.+\phi_j\sum_\alpha \phi_\alpha\cdot [\alpha,\alpha,j,1,\dots,\ell]\right)\prod f_\alpha +O(\ep^{s+3})\,,
\end{aligned}
\end{equation}
where as usual the hat denotes omission of the index. In what follows we will actually only need to use the formulas above for $\ell=3$, but setting $\ell=3$ does not simplify the formula above much, so we have chosen to give the general expression.

\smallskip
For investigating the rank of the Hessian of the theta function we will need to compute the second order derivatives of the theta function. While these formulas can also be written for arbitrary~$\ell$, in our arguments we will only need $\ell=2$, and here this makes the formulas much shorter. So from now on we take $m=m_0=\chars{110\dots0}{110\dots0}$, and first compute for $J=1,2$
$$
\partial_{\tau_{JJ}}\theta_{m_0}=\frac{\prod f_\alpha}2\cdot\left(\phi_J\cdot X_2+\phi_J\cdot Y_2+\psi_J\cdot X_2\right)+O(\ep^3)\,,
$$
where the extra factor of $\frac12$ is due to $\frac{1}{4\pi i}$ appearing instead of $\frac{1}{2\pi i}$ in the heat equation, for differentiating with respect to~$\tau_{JJ}$. Next, we compute
$$
\partial_{\tau_{12}}\theta_{m_0}=\prod f_\alpha\cdot\left(1+\sum\phi_\alpha\cdot[\alpha,\alpha]+\psi_1\cdot [1,1]+\psi_2\cdot [2,2]\right)+O(\ep^2)\,;
$$
$$
 \partial_{\tau_{1j}}\theta_m=\phi_j\cdot [j,2]\cdot\prod_\alpha f_\alpha+O(\ep^2)\,;
$$
$$
 \partial_{\tau_{2j}}\theta_m=\phi_j\cdot [j,1]\cdot\prod_\alpha f_\alpha+O(\ep^2)\,;
$$
$$
 \partial_{\tau_{jk}}\theta_m=\phi_j\cdot\phi_k\cdot [j,k,1,2]\cdot\prod_\alpha f_\alpha +O(\ep^3)\,,
$$
and finally
$$
\begin{aligned}
 \partial_{\tau_{jj}}\theta_m&=\frac{\prod f_\alpha}2\cdot\left([1,2]+\psi_j\cdot [j,j,1,2]+\phi_j\cdot \sum_\alpha\phi_\alpha\cdot [\alpha,\alpha,1,2]\right)+O(\ep^3)\\
 &=\frac{\prod f_\alpha}2\cdot\left(\phi_j\cdot (X_2+Y_2)+\psi_j\cdot [j,j,1,2]\right)+O(\ep^3)\,,
\end{aligned}
$$
where of course in each of these formulas further terms in the expansion can also be easily written down, though the formulas become very lengthy.

We give an example: the principal $4\times 4$ minor of the Hessian of $\theta_{m_0}$, formed by the rows and columns numbered $1,2,j,k$ (for $3\le j<k\le g$) is given by $\prod f_\alpha$ multiplied by
\begin{equation}\label{hess2}
\begin{smallmatrix}
\frac12\phi_1\cdot (X_2+Y_2)&1+\sum_{\alpha,\beta}\phi_\alpha\phi_\beta\cdot[\alpha,\alpha,\beta,\beta]&\phi_j\cdot[2,j]
+\phi_j\sum\phi_\alpha\cdot[\alpha,\alpha,2,j]&\phi_k\cdot[2,k]+\phi_k\cdot\sum\phi_\alpha\cdot[\alpha,\alpha,2,k]\\
*&\frac12\phi_2\cdot (X_2+Y_2) &\phi_j\cdot[1,j]+\phi_j\sum\phi_\alpha\cdot[\alpha,\alpha,1,j]&\phi_k\cdot[1,k]+\phi_k\sum\phi_\alpha\cdot[\alpha,\alpha,1,k]\\
*&*&\frac12\phi_j\cdot (X_2+Y_2)+\frac12\psi_j\cdot[j,j,1,2]&\phi_j\cdot\phi_k\cdot[j,k,1,2]\\
*&*&*&\frac12\phi_k\cdot (X_2+Y_2)+\frac12\psi_k\cdot[k,k,1,2]
\end{smallmatrix}
\end{equation}
where the $*$'s below the diagonal signify the fact that the matrix is symmetric, and we have expanded the matrix up to dropping the $\ep^3$ terms.

\section{The local form of the hyperelliptic locus}\label{sec:hyp}
In this section we describe the hyperelliptic locus near $\II_g$, proving~\Cref{thm:hypfull}, which is a slightly stronger version of~\Cref{thm:hypsmall}.
We first check that using the special fundamental system~\eqref{eq:specsystem}, the resulting irreducible component $\HJ_g^{I}$ of $\HJ_g\subset\HH_g$ given by~\Cref{prop:hyponecomponent} contains the locus~$\II_g$ of diagonal period matrices. This follows from the following
\begin{lm}\label{lm:vanish}
If $m\in\calE$ can be written as a sum of strictly less than~$g$ among the characteristics $o_1,\dots,o_g,e_2,\dots,e_{g+2}$ of the special fundamental system~\eqref{eq:specsystem}, then $m+b^g\in\calE^*$.
\end{lm}
\begin{proof}
We proceed by induction on $g$, with the lemma being easily true for $g=1, 2$. Then observe that deleting the characteristics $o_g$ and $e_{g+2}$ from the chosen special fundamental system, and forgetting the $g$-th column of each characteristic gives the special fundamental system of genus $g-1$. Note now that the last column of $b^g$ is equal to $\chars11$. Thus unless the sum for $m$ includes one of the characteristics that have a non-zero $g$'th column, the characteristic $m+b^g$ also has the last column $\chars11$, and thus does not lie in $\calE_0$. The only three characteristics among the special fundamental system that have a non-zero $g$'th column are $o_g,e_{g+1},e_{g+2}$. If $m$ is a sum of characteristics including $o_g$, then we observe that $b^g+o_g= b^{g-1}\oplus\chars{0}{0}$, and thus the statement follows from the inductive assumption, by ignoring the $g$'th column of characteristics (since if any of the first $g-1$ columns of a genus~$g$ characteristic is equal to $\chars11$, it is not in $\calE_0$, and we are now using one less characteristic in the sum for $m$). If $m$ is a sum of characteristics not including $o_g$, but including exactly one of $e_{g+1}$ or $e_{g+2}$, we note that $b^g+e_{g+1}= b^{g-1}\oplus\chars{0}{1}$ and $b^g+e_{g+2}= b^{g-1}\oplus\chars{1}{0}$, and the same argument applies as for the previous case of~$o_g$. Finally, if both $e_{g+1}$ and $e_{g+2}$ are used in this sum representation, we note that since the bottom characteristic of $e_{g+1}+e_{g+2}$ is equal to zero vector in $\ZZ_2^g$, we can again proceed by induction.
\end{proof}
\begin{cor}
All theta constants $\theta_{m+b^g}$ vanishing along the hyperelliptic component $\HJ_g^{I}$ also vanish along~$\II_g$.
\end{cor}

We are now ready to prove the most precise version of our result on the hyperelliptic locus.
\begin{thm}\label{thm:hypfull}
For any irreducible component $\XX$ of the hyperelliptic locus $\HJ_g\subset\HH_g$, such that $\XX\supset\II_g$, the tangent space to~$\XX$ at any point of $\II_g$ is the set of period matrices satisfying the set of equations $\lbrace\tau_{\pi(i)\pi(j)}= 0\rbrace_{\forall 1\le i,j\le g, |i-j|>1}$, where $\pi\in S_g$ is the permutation that is the image in $S_g$ of the element $\sigma\in Stab_{\II_g}$ that sends $\HJ_g^I$ to $\XX$.
\end{thm}
\begin{proof}
We first use \Cref{lm:vanish} to prove~\Cref{thm:hypsmall}, showing that the tangent space to $\HJ^I_g$ at any point of $\II_g$ is given by equations $\tau_{jk}=0$ for all $|j-k|>1$.

Indeed, recall that $b^g=o_1+\dots+o_g$, and consider a characteristic of the form
$$
 m=b^g+o_1+\dots+\widehat {o_{j_1}}+\dots+\widehat {o_{j_2}}+\dots+\widehat {o_{j_3}}+\dots+o_g=o_{j_1}+o_{j_2}+o_{j_3}\,,
$$
with $1\le j_1<j_2< j_3\le g$. From the first expression for~$m$ it follows that $\theta_m$ vanishes identically on $\HJ^I_g$. Writing the columns of~$m$ as $\chars{\ep_i}{\de_i}$, we compute that 
$$
 \ep_{j_1}=\ep_{j_2}=\ep_{j_3}=\de_{j_1}=\de_{j_3}=1
$$
and that all other $\ep_i$ and $\de_i$ are zero (note that this includes $\de_{j_2}=0$). Thus $m\in\calE_2$, in agreement with~\Cref{lm:vanish}, and in particular the lowest order term of the expansion of $\theta_m$ near $\II_g$ is linear, given explicitly by~\eqref{eq:thetaexpand} (where we use $\ell=2$, see the genus 4 example there) as
\begin{equation}\label{eq:linearterm}
 \theta_m(\tau)=\frac{1}{2\pi i}\tau_{j_1j_3}\vt_{j_1}'\vt_{j_3}'\prod_{h\ne j_1,j_3}\vt_h
\end{equation}
in our conventions. Going over all possible choices of $1\le j_1<j_2<j_3\le g$ means that in the above expressions there appear all $\tau_{j_1j_3}$ such that $j_3-j_1>1$. Recalling that we are working within the space of symmetric matrices, it follows that the tangent space to $\HJ^I$ at a point in~$\II_g$ is contained in the locus of tridiagonal matrices given by equations $\lbrace\tau_{jk}=0\rbrace_{\forall |k-j|>1}$. Since the hyperelliptic locus $\calHJ_g$ is of dimension $2g-1$, any irreducible component of its preimage $\HJ_g$ is also of dimension $2g-1$. Thus the tangent space to $\HJ^I_g$ at any point of $\II_g$ must be of dimension at least $2g-1$. Since this tangent space is contained in the space of tridiagonal matrices, which also has dimension $2g-1$, it must be equal to it, and~\Cref{thm:hypsmall} is thus proven.

\smallskip
We now combine this with the study of the action of $\Gamma_g$ on irreducible components of $\HJ_g$ to obtain the full statement. By~\Cref{lm:tsu}, for any irreducible component $\XX$ of $\HJ_g$ containing $\II_g$, there exists $\sigma\in Stab_{\II_g}$ such that $\sigma(\HJ^I_g)=\XX$. Thus the tangent space to $\XX$ is the image of the tangent space to $\HJ^I_g$ under the action of~$\sigma$. Since clearly the action of $\Gamma_1$ on each column of the characteristic does not change the linear term of the expansion~\eqref{eq:linearterm}, it follows that the action of $\sigma$ on the linear equations for the tangent space to $\HJ^I_g$ along $\II_g$ is simply by permuting the columns of the period matrix according to the image of the permutation $\sigma$ under the surjection $Stab_{\II_g}\to S_g$.
\end{proof}
\begin{rem}
From the proof of the theorem we see that multiple irreducible components of~$\HJ_g$ contain~$\II_g$, and many of them may have the same tangent space along~$\II_g$. This can already be seen in the first interesting case of~$g=3$. Recall that $\HJ_3$ has 36 irreducible components, each one being the zero locus of one of the 36 even genus 3 theta constants. Nine of these 36 components, those with characteristics in $\calE^*=\calE_2$, contain $\II_3$: namely these correspond to characteristics 
$$ 
 \chars{110 }{110 }, \,\, \chars{111 }{110 }, \,\, \chars{110 }{111 },\,\, \chars{101 }{101 }, \,\, \chars{111 }{101 }, \,\, \chars{101 }{111 },\,\, \chars{011 }{011 }, \,\, \chars{111 }{011 }, \,\, \chars{011 }{111 }\,.
$$
The irreducible components of $\HJ_3$ corresponding to each of the first three characteristics have local equation $\tau_{12}=0$ near $\II_3$; the next three components have local equation $\tau_{13}=0$, and the final three cases have local equation $\tau_{23}=0$.

For the next, more interesting, case of $g=4$, a similar analysis can be performed, noting that here 10 even theta constants vanish on every irreducible component of $\HJ_4$. As a side remark, notice that no matter how the columns are permuted, the three local defining equations for the tangent space of a component of $\HJ_4$ near $\II_4$ cannot be of the form $ \tau_{12}= \tau_{13}= \tau_{14}=0$. Indeed, this corresponds to the fact that the locus of products $\HH_1\times\HH_3$ (which locally along $\II_g$ is given by precisely these three equations) is not contained in $\HJ_4$, as $\HJ_3\subsetneq\HH_3$.

In arbitrary genus, we see that the only element of $S_g$ that fixes the set of equations $\lbrace \tau_{jk}=0\rbrace_{\forall\, |j-k|>1}$, is the product of transpositions $\pi= ( 1,g) ( 2, g-1)\dots ( \frac{g+3}{2}, \frac{g}{2})$. Thus there are altogether $g!/2$ possible different tangent spaces along $\II_g$ to the different irreducible components of $\HJ_g$ containing $\II_g$.
\end{rem}

\section{The theta-null divisor and the vanishing gradient locus}\label{sec:gn}
We now proceed to investigate the loci $\vtn$ and $\vgn$ locally near~$\calD_g$. The method we use here, and then also for dealing with the Hessian rank loci, will be different, and more general, than what we have done for the hyperelliptic locus. 

The outline of our argument is as follows. First we use the expansion of theta functions near $\II_g$, as computed in~\Cref{sec:expand}, to determine the dimensions of the tangent spaces near $\II_g$ to the loci given by equations $\theta_{m_0}(z)=0$ (resp.~$\grad\theta_{m_0}(z)=0$), where $m_0$ is the simplest even (resp.~odd) characteristic, i.e.~a characteristic that lies in $\calE_2$ (resp.~$\calO_3$) for which the expansions are given. This is done by intersecting the Taylor expansions of the defining equations with a suitable ``transverse'' subvariety, to simplify computations. 

As one can already see from the formulas in \Cref{sec:expand}, and from the proofs of these statements, such expansions near $\II_g$ for characteristics in $\calE_\ell$ or $\calO_\ell$ with $\ell\gg 2$ can become very complicated --- essentially as the corresponding locus contains the diagonal $\II_g$ with high multiplicity. Thus to deal with arbitrary characteristics, which is necessary to understand the loci $\vtn,\vgn\subset\calA_g$, we will act by $\Gamma_g$, and will have to use the fact that the loci of interest contain the ``big'' diagonal $\LL_g^e$ or $\LL_g^o$ defined below, which consists of block-diagonal matrices that have two to four $1\times 1$ blocks, and the remaining blocks are $2\times 2$. Since the setwise stabilizer $Stab_{\LL_g}$ contains suitable subgroups of $G_g$, which by \Cref{lm:orbita} permute all even (resp.~odd) characteristics, it will turn out that the statement for an arbitrary characteristic~$m$ will reduce to the statement for~$m_0$ --- and \Cref{prop:maxcomponent} is a general statement to this effect. 

\subsection{Local structure of $\tnmz$ and $\gnmz$ near the diagonal}
We will start by explicit computations for the characteristics in $\calE_2$ and $\calO_3$. The even case is straightforward, using the first term of the Taylor expansion.
\begin{prop}
Denote $m_0:=\chars{110\dots0}{110\dots0}\in\calE_2$. Then in a sufficiently small neighborhood of $\II_g$, the locus 
$$
 \tnmz=\lbrace \tau\in\HH_g\colon \theta_{m_0}(\tau)=0\rbrace
$$
is smooth, of codimension~$1$ in~$\HH_g$.
\end{prop}
\begin{proof} 
Since we are dealing with one non-trivial equation, it is clear that the locus is of codimension one, so the point of the statement is smoothness. For this, we observe that by~\eqref{eq:thetaexpand} the local defining equation admits the expansion 
$$
 \theta_{m_0}(\tau)= c \tau_{12} + O( \ep^2)
$$ 
for some non-zero $c$. Thus clearly the zero locus is smooth, of codimension one.
\end{proof}
Of course by acting by $Stab_2$ we obtain the same statement for $\tnm$ for any $m=\chars\ep\de\in\calE_2$ --- and in this case the local lowest order defining equation is $\tau_{ab}=0$, where $a$ and $b$ are the two $\chars11$ columns of the characteristics..

\smallskip
The odd case is much more elaborate, as there are~$g$ components of the gradient giving the~$g$ defining equations of $\vgn$, and we want to check that these equations are locally independent. This is hard to do directly, and for bounding the dimension of local irreducible components of various loci near $\II_g$, we will use the following well-known statement.
\begin{lm}\label{lm:cutexpand}
Let $X=\Spec \left(\CC[x_1,\dots,x_N]/\langle F_1,\dots,F_k\rangle\right)$ be an irreducible affine subscheme of $\AA^N=\Spec\,\CC[x_1,\dots,x_N]$, of dimension $n=\dim X$. Suppose $x\in X\subset\AA^N$, and let $\frakM_x\subset\calO_{\AA^N,x}$ be the maximal ideal of the local ring. For any $h\ge 1$ let $N(h)$ denote the number of algebraically independent among the images of $F_1,\dots,F_k$ modulo $(\frakM_x)^h$. Then the inequality $n\le N-N(h)$ holds for any $h$.
\end{lm}
We note that of course if $X$ is a not necessarily irreducible affine scheme, the lemma shows that every irreducible component of $X$ that contains $x$ must have dimension at most $N-N(h)$.
\begin{proof}
We denote $\frakm_x\subset\calO_{X,x}$ the maximal ideal; since $X$ is $n$-dimensional, we of course have $\dim\left(\calO_{X,x}/(\frakm_x^h)\right)=O(n^h)$ as $h\to\infty$.

Note now that the map $\calO_{\AA^N,x}\to \calO_{X,x}$ maps the ideal $\langle F_1,\dots,F_k\rangle$ to $0$, and maps $\frakM_x$ onto $\frakm_x$. Thus we have the bound
$$
 \dim\left(\calO_{\AA^N,x}/\langle F_1,\dots,F_k,(\frakM_x)^h\rangle\right)\ge\dim\left(\calO_{X,x}/(\frakm_x)^h\right)\,,
$$
while the map of the completions
$$
 \widehat{\calO_{\AA^N,x}}/\langle F_1,\dots,F_k\rangle\to\widehat{\calO_{X,x}}
$$
is surjective. Since the function $N(h)$ is monotone, and bounded above by $N$, it must have a limit, and from the above inequalities we thus obtain
$$
n\le N-\lim_{h\to\infty} N(h)\le N-N(h)
$$
for any $n$.
\end{proof}

We now obtain the dimension result for odd characteristics in $\calO_3$. 
\begin{prop}
Denote $m_0:=\chars{1110\dots0}{1110\dots0}$. Then in a sufficiently small neighborhood of $\II_g$, the locus 
$$
 \gnmz=\lbrace \tau\in\HH_g\colon \grad \theta_{m_0}(\tau)=0\rbrace
$$
is of codimension~$g$ in~$\HH_g$.
\end{prop}
\begin{proof}
We will apply \Cref{lm:cutexpand} for $X=\gnmz$, and consider the intersection of $X$ with the locus $Y$ given by equations
$$
 \lbrace \tau_{jk}=0 \quad \forall\, 4\le j<k\le g\rbrace\qquad\hbox{and}\qquad \lbrace\tau_{1j}=\tau_{2j}=\tau_{3j}\quad \forall\, 4\le j\le g\rbrace\,.
$$
These are altogether
$$
 \frac{(g-3)(g-4)}{2}+2(g-3)=\frac{g(g-3)}{2}
$$
equations, which are clearly independent, and all of which are of course satisfied on $\II_g$. Thus $\II_g\subset Y$, and $\codim_{\HH_g}Y=\tfrac{g(g-3)}{2}$. We claim that in a sufficiently small neighborhood $U$ of $\II_g$ we have $X\cap Y\cap U=\II_g$. Since $\dim\II_g=g$, this would imply that
$$
 \dim (X\cap U)\le \dim\II_g+\codim_{\HH_g}Y=g+\frac{g(g-3)}{2}=\frac {g(g-1)}{2}\,,
$$
and thus that the codimension of $X\cap U$ in $U$ is at least $g$. Since $X\subset\HH_g$ is given by $g$ equations, it follows that $\codim_{\HH_g}X\cap U=g$.
 
To see that $X\cap Y\cap U=\II_g$, we simply plug in the defining equations of $Y$ into the expansion~\eqref{eq:expgrad1} of the theta gradient computed above, always excluding the common factor of $\prod f_\alpha$, which is non-zero at a generic point of $\II_g$:
$$
\partial_{z_1}\theta_{m_0}\equiv [2,3]+\sum_\alpha\phi_\alpha\cdot[\alpha,\alpha,2,3]\equiv \tau_{23}+\phi_1\tau_{12}\tau_{13}+\sum_{j\ge 4}\phi_j\tau_{2j}\tau_{3j}\mod\frakm^3\,,
$$
where we recall that $t_1,\dots,t_g$ are considered fixed, so that $\frakm=\langle\left\lbrace\tau_{ab}\right\rbrace_{1\le a<b\le g}\rangle$.
Now substituting into this the defining equations for $Y$ we obtain
$$
\partial_{z_1}\theta_m|_Y\equiv \tau_{23}+\phi_1\tau_{12}\tau_{13}+\sum_{j\ge 4}\phi_j\tau_{1j}^2\mod\frakm^3\,,
$$
and of course the expressions for $\partial_{z_2}\theta_m|_Y$ and $\partial_{z_3}\theta_m|_Y$ are completely analogous:
$$
\partial_{z_2}\theta_m|_Y\equiv \tau_{13}+\phi_1\tau_{12}\tau_{23}+\sum_{j\ge 4}\phi_j\tau_{1j}^2\mod\frakm^3\,,
$$
$$
\partial_{z_3}\theta_m|_Y\equiv \tau_{12}+\phi_1\tau_{13}\tau_{23}+\sum_{j\ge 4}\phi_j\tau_{1j}^2\mod\frakm^3\,.
$$
For the partial $z$-derivatives in the other directions we obtain from~\eqref{eq:expgrad2}
$$\begin{aligned}
\partial_{z_j}\theta_m&\equiv\phi_j\cdot ([j,1,2,3]+\sum\phi_\alpha\cdot[\alpha,\alpha,j,1,2,3])+\psi_j\cdot[j,j,j,1,2,3]\\
&\equiv\phi_j\cdot\left(\tau_{1j}\tau_{23}+\tau_{2j}\tau_{13}+\tau_{3j}\tau_{12}+\sum\phi_\alpha\cdot[\alpha,\alpha,j,1,2,3]\right)
+\psi_j\tau_{1j}\tau_{2j}\tau_{3j}\mod\frakm^4\,.
\end{aligned}
$$
We compute $[1,1,j,1,2,3]|_Y=\tau_{12}\tau_{13}\tau_{1j}$, $[j,j,j,1,2,3]|_Y=\tau_{1j}^3$, and
$$
[k,k,j,1,2,3]|_Y=\tau_{1j}\tau_{2k}\tau_{3k}+\tau_{2j}\tau_{1k}\tau_{3k}+\tau_{3j}\tau_{1k}\tau_{2k}=3\tau_{1j}\tau_{1k}^2\,,
$$
so that we can finally evaluate
$$
\begin{aligned}
\partial_{z_j}\theta_m|_Y&\equiv \phi_j\tau_{1j}\cdot\left(\tau_{23}+\tau_{13}+\tau_{12}+\phi_1\tau_{12}\tau_{13}+\phi_2\tau_{12}\tau_{23}+\phi_3\tau_{13}\tau_{23}\right)\\
&+3\phi_j\tau_{1j}\sum_{k\ge 4}\phi_k\cdot\tau_{1k}^2+\psi_j\tau_{1j}^3\mod\frakm^4\,.
\end{aligned}
$$

We are interested in the locus $Y\cap\gn\chars\ep\de$, and thus we can substitute the expansions of the equations $\partial_{z_1}\theta_m|_Y=\partial_{z_2}\theta_m|_Y=\partial_{z_3}\theta_m|_Y=0$ in the last equation, obtaining
$$
\begin{aligned}
 \partial_{z_j}\theta_m|_Y&\equiv \phi_j\tau_{1j}\cdot(\partial_{z_1}\theta_m|_Y+\partial_{z_2}\theta_m|_Y+\partial_{z_3}\theta_m|_Y)\\ &+\tau_{1j}^3(\psi_j-2\phi_j^2)
\mod\frakm^4\,.
\end{aligned}
$$
Since $\psi_j-2\phi_j^2$ is a not identically zero function on $\HH_1$, it does not vanish at a generic point of~$\II_g$, and thus the vanishing of $\grad \theta_{m_0}\mod\frakm^4$ implies $\tau_{1j}=0$ for all $4\le j\le g$, which together with the defining equations for $Y$ implies that $\tau$ is diagonal except possibly for the entries $\tau_{12},\tau_{23},\tau_{13}$. However, the vanishing of $z_1,z_2,z_3$ derivatives $\mod\frakm^2$ gives the equations $\tau_{12}=\tau_{23}=\tau_{13}$. Thus altogether it follows that every irreducible component of $\gnmz$ that contains $\II_g$ has codimension at least~$g$ in $\HH_g$, and thus codimension precisely~$g$. 
\end{proof}

\subsection{Theta-null and gradient loci for arbitrary characteristics}
Note that the computations above for $m_0\in\calO_3$ are already quite involved. If we wanted to deal with the locus $\theta_m(\tau)=0$ or $\grad \theta_m(\tau)=0$ for $\ell\gg 3$, the computation would be daunting, as we would need to consider the Taylor expansion to order $\ell/2$, see eg.~\eqref{eq:thetaexpand}. Thus instead we will use the action of $G_g$, and the fact that the loci we are interested in contain the big diagonal, defined below.
 
The loci we deal with will be defined in terms of the geometry of the theta divisor. While just thinking of the locus of ppav such that the theta divisor satisfies some geometric condition (eg has a singular point with some properties) only defines a locus {\em set-theoretically} as a {\em subset} of $\calA_g$, of course the loci we are interested in are in fact algebraic. However, thinking of them as {\em subvarieties} of $\calA_g$, by arguing that the conditions that define them are {\em algebraic} is also insufficient for our purposes. Indeed, for us it is important to consider these as {\em subschemes} of $\calA_g$ --- as it is for Mumford's computation \cite{mumforddimag} of the class of the Andreotti-Mayer divisor, and in general for thinking about the Andreotti-Mayer loci.

The way we think of the scheme structure on these loci is as follows: recall that the universal cover of the universal family $\calX_g\to\calA_g$ of ppav is $\HH_g\times\CC^g$, with the covering group being the semidirect product $\Sp(2g,\ZZ)\rtimes \ZZ^{2g}$. The theta function is a global holomorphic function on $\HH_g\times\CC^g$, and various geometric conditions on the singularities of the theta divisor can then be interpreted as various analytic equations on $\HH_g\times\CC^g$ involving the theta function and its partial derivatives. 

For the purposes of obtaining the results below, we will only be interested in the geometry of the theta divisor at the two-torsion points of the ppav, and these conditions can be defined analytically over $\HH_g$ in terms of theta constants with characteristics and their derivatives. The loci thus defined naturally come with an analytic defining ideal on $\HH_g$; the defining ideal is invariant under the action of $\Gamma_g$ (or $\Gamma_g(2)$ depending on the context), and thus images of these in $\calA_g$ (or $\calA_g(2)$) have natural defining algebraic equations, and thus a natural scheme structure. This explains our care in describing the following general setup.

For any $m=\chars\ep\de\in\calE $ we set $\tnm:=\tn{\chars\ep\de}$, and for any $m=\chars\ep\de\in\calO$ we set
$\grad\tnm:=\grad\tn{\chars\ep\de}$. Recall that by definition $\tn=\cup_{m\in\calE}\tnm$. 

Given any analytic subspace $\XX\subset\HH_g$, which will be contained in either $\tn$ or $\gn$, depending on the context, we decompose it as $\XX=\cup_{m} \XX_m$, where $\XX_m:=\XX\cap\tnm$ (or respectively $\XX_m:=\XX\cap\grad\tnm$). We note that $\XX_m\cap\XX_n$ may be non-empty.

Assume $\XX\subset\tn$ (or $\XX\subset\gn$) is an analytic subspace of~$\HH_g$ satisfying $\Gamma_g\circ\XX=\XX$ (as a set), so in particular $\Gamma_g$ acts transitively on the set of $\XX_m$ for all $m\in\calE$, and for any $m$ the setwise stabilizer $\Gamma_m$ of $\XX_m$ contains $\Gamma_g(2)$. Denoting $X:=p(\XX)\subset\calA_g$ the image, observe that $p^{-1}(X)=\XX$, and there exists a well-defined scheme structure on $X$ induced by the defining equations of~$\XX$.

We will be interested in computing the dimension of irreducible components of $X$ containing $\calD_g$. This is related to computing the dimension of irreducible components of $\XX$ containing $\II_g$, which we will approach via Taylor expansions in the neighborhood of $\II_g$. The difficulty is that a priori the scheme $X$ may have embedded components containing $\calD_g$, and thus thinking of $X$ as a subvariety may not suffice. Essentially the difficulty is that if $\II_g\subset\XX_m$, and we can determine irreducible components of $\XX_m$ containing $\II_g$, it could be that also $\II_g\subset\XX_n$ for some other $n$, and the image in $\calA_g$ of an irreducible component of $\XX_n$ containing $\II_g$ may be strictly contained in the image in $\calA_g$ of an irreducible component of $\XX_m$ containing $\II_g$. We will deal with this by explicitly imposing the additional assumption that a component contains the big diagonal that we now define (this condition will hold for all those loci that we are interested in). 

For the even case (i.e.~when we are interested in $\XX\subset\tn$), we define the big diagonal as
$$
 \LL_g^{e}:=\HH_1 \times \HH_1 \times \HH_2 \times\dots \times \HH_2 \quad {\rm or} \quad \LL_g^{e}:=\HH_1 \times \HH_1 \times \HH_2 \times\dots \times \HH_2\times \HH_1\,,
$$
(where the presence of the last factor depends on the parity of~$g$) as the locus of period matrices that have one less than the maximal possible number of $2\times 2$ blocks along the diagonal. The direct product $L_g^{e}:= \Gamma_1 \times\Gamma_1\times \Gamma_2\times \dots\times\Gamma_2$ or $L_g^{e}:=\Gamma_1 \times\Gamma_1\times \Gamma_2\times \dots\times\Gamma_2\times\Gamma_1$ (where the presence of the last factor depends on the parity of~$g$) is clearly contained in the stabilizer $Stab_{\LL_g^{e}}$. Finally, in the even case we set $m_0:= \chars{110\dots0}{110\dots0}\in\calE_2$, so that $\tnmz\supset \II_g$. 

Similarly for the odd case of $\XX\subset\gn$ we define the big diagonal to be $\LL_g^{o}:= \HH_1 \times \LL_{g-1}^{e}$, and note that its stabilizer contains the direct product $L_g^{o}:= \Gamma_1\times L_g^{e}$. In this case we set $m_0:= \chars{1110\dots0}{1110\dots0}\in\calO_3$, so that again $\gnmz\supset \II_g$.
 
We will use these to investigate irreducible components locally near $\II_g$ by applying the following statement.
\begin{prop}\label{prop:maxcomponent} 
Let $\XX\subset \tn$ (resp.~$\XX\subset\gn$) be an analytic subspace of $\HH_g$ containing $\II_g$, such that~$\XX=\cup \XX_m$ is invariant under $\Gamma_g$. Let $\YY\subset \XX_{m_0}$ be an irreducible component of $\XX_{m_0}$ containing $\II_g$.
 
If $\LL_g^e\subset\YY$ (resp. $\LL_g^o\subset\YY$), then for each $m\in \calE^*$ (resp.~$m\in\calO^*$) there exists an element $\sigma_m\in G_g$ mapping $m_0$ to~$m$, such that $\sigma_m (\YY)$ is an irreducible component of $\XX_m$ containing $\II_g$ .
\end{prop}
\begin{proof} 
We give the argument for the even case; the argument for the odd case being completely analogous, using $\LL_g^o$ and $L_g^o$ instead of $\LL_g^e$ and $L_g^e$.

We first observe that if $m\in\calE_2$ (resp.~$m\in\calO_3$), then, since $Stab_{\II_g}$ acts transitively on $\calE_2$ (resp.~$\calO_3$), 
there exists $\sigma_m\in Stab_{\II_g}$ sending $m_0$ to $m$. Since $\sigma_m(\II_g)=\II_g$ (as a set), the image $\sigma_m(\YY)$ contains $\II_g$. Since $\sigma_m(\XX_{m_0})=\XX_m$ by the $\Gamma_g$-invariance of $\XX$ and by the definition of $\XX_m$, it follows that $\sigma_m(\YY)$ is an irreducible component of $\XX_m$, containing~$\II_g$.

To deal with the case of $m\in\calE_\ell$ with $\ell\ge 4$, we first observe that since $Stab_{\II_g}$ acts transitively on $\calE_\ell$, it is enough to deal with the case of $m=\chars{1\dots10\dots0}{1\dots10\dots0}$ with $\ell$ columns equal to $\chars11$. By the proof of \Cref{lm:orbita} there exists an element $\sigma _m \in L_g^e\subset Stab_{\LL_g^e}$ such that $\sigma_m\cdot m_0=m$. Since $\LL_g^e\subset\YY\subset\XX_{m_0}$ by assumption, and since $\XX$ is $\Gamma_g$-invariant, it follows that $\sigma_m(\LL_g^e)=\LL_g^e\subset \sigma_m(\YY)\subset\XX_m$, and in particular $\sigma_m(\YY)\supset\LL_g^e\supset\II_g$. If $\sigma_m(\YY)$ were not an irreducible component of $\XX_m$, i.e. if there existed an irreducible component $\WW$ of $\XX_m$ strictly containing $\sigma_m(\YY)$, then by invariance of $\XX$ under $\Gamma_g$, the preimage $\sigma_m^{-1}(\WW)$ would be an irreducible component of $\XX_{m_0}$ containing $\YY$, giving a contradiction. 
\end{proof}
The proposition can be rephrased as a statement on subvarieties of $\calA_g$:
\begin{cor}\label{cor:maxcomponent}
 Let $\calX\subset \vtn\subset\calA_g$ (resp.~$\calX\subset\vgn\subset\calA_g$) be an algebraic subvariety containing $p(\LL_g^e)\supset \calD_g$ (resp.~$p(\LL_g^o)\supset \calD_g$). Denote $\XX:=p^{-1} (\calX)\subset \HH_g$, and let $\YY\subset\XX_m:=\XX\cap\tnm$ (resp.~$\XX\cap\gnm$) be an irreducible component containing $\LL_g^e$ (resp.~$\LL_g^o$). Then $p(\YY)$ is an irreducible component of $\calX$ containing $\calD_g$.
\end{cor}

What this proposition essentially rules out is the situation discussed above, where $\II_g\subset \YY\subset\XX_{m_0}$, but where also $\II_g\subset \WW\subset \XX_m$ such that $p(\YY)\subsetneq p(\WW)\subset\calA_g$. Notice that since $\II_g\subset\XX_m$ if and only if $m\in\calE^*$ (resp.~$\calO^*)$, the characteristics $m\in\calE^0$ (resp.~$\calO^1)$ do not occur in the above discussion.

\begin{rem} \label{rem:hyp} 
The above proposition also holds in a more general context. For example, let $M=(m_1, \dots, m_k)$ be a sequence of even characteristics, and let $\calE_{M}$ be the set of all ordered $k$-tuples of characteristics that form the $\Gamma_g$ orbit of $M$. Then we can also decompose $\XX\subset\HH_g$ as
$$\XX=\cup_{(n_1, \dots, n_k)\in \calE _{M}} \XX_{n_1, \dots, n_k}\,,$$
where $\XX_{n_1, \dots, n_k}:=\XX\cap\theta_{n_1,\,\rm null}\cap\dots\cap \theta_{n_k,\,\rm null}$. In this case we obtain similar statements under the assumption $\LL_g \subset \XX_{(m_1, \dots, m_k)}$. In particular this applies to the hyperelliptic case discussed in the previous section, giving an alternative approach to the results there.
\end{rem}
We will now apply this setup for the vanishing theta gradient loci.  

\begin{proof}[Proof of {\Cref{thm:Gg}}]
We apply~\Cref{prop:maxcomponent} for 
$$
\XX=\gn=\cup_m\gnm\subset\HH_g\,.
$$
To avoid confusion, denote $n_0:=\chars {110\dots0}{110\dots0}$ the even characteristic in genus $g-1$. Note that the locus $\YY=\HH_1\times\theta_{\rm n_0,\, null}\subset \HH_1\times \HH_{g-1}$ is irreducible and has codimension $g$ in $\HH_g$, thus is an irreducible component of $\gnmz$, containing $\LL_g^o$. Then \Cref{cor:maxcomponent} implies that $p(\YY)=\calA_1\times\vtn$ is an irreducible component of $\vgn$, which contains $\calD_g$ (and in fact of course contains $p(\LL_g^o)$).

For components containing the hyperelliptic locus, suppose $\YY\subset\gnm$ is an irreducible component of $\gnm$ containing some component $\XX$ of the hyperelliptic locus $\HJ_g$, such that $\XX\supset\II_g$. By~\Cref{lm:tsu}, there exists $\sigma\in\Gamma_g$ which lies in $Stab_{\II_g}$ and maps $m$ to $m_0$. We can now apply~\Cref{cor:maxcomponent} for $\sigma(\YY)\subset\gnmz$, yielding the result.
\end{proof}
 
\section{The hessian rank loci $\vtn^2$ and $\vtn^3$}\label{sec:tn}
In this section we investigate the geometry of the loci with given rank of the Hessian of the theta function near~$\II_g$, proving Theorems~\ref{thm:tn2} and ~\ref{thm:tn3}. Using~\Cref{cor:maxcomponent}, at the end of the day it will suffice to study the hessian of~$\theta_{m_0}$ near~$\II_g$.

\subsection{The rank two locus}
Our first goal is to prove~\Cref{thm:tn2}, that $\calA_1\times\calA_{g-1}$ is an irreducible component of $\vtn^2$. By applying~\Cref{cor:maxcomponent}, it will suffice to show that $\HH_1\times\HH_{g-1}$ is an irreducible component of $\tnmz^2$. For this, similarly to how we dealt with the locus $\vgn$, working locally near $\II_g$ we will compute the intersection of $\tnmz^2$ with the locus $Z\subset\HH_g$ given by the equations $\tau_{jk}=0$ for all $2\le j<k\le g$ and $\tau_{1j}=\tau_{2j}$ for all $3\le j\le g$. Note that $(\HH_1\times\HH_{g-1})\cap Z=\II_g$ by definition. Since $\HH_1\times\HH_{g-1}\subset\tnmz^2$, the following proposition will suffice to prove~\Cref{thm:tn2}. 
\begin{prop}
For a sufficiently small neighborhood $U$ of $\II_g$ the equality $\tnmz^2 \cap Z\cap U=\II_g$ holds.
\end{prop}
\begin{proof}
Indeed, we will plug in the defining equations of $Z$ into the defining equations of $\tnmz^2$ and check that they imply $\tau_{1j}=0$ for all $1<j\le g$. To see this, it will be sufficient to consider the principal $3\times 3$ minors of the Hessian of $\theta_{m_0}$ that include the first and second rows. Using the expansion~\eqref{hess2}, for the $3\times 3$ principal minor obtained by taking rows and columns $1,2,j$ for some $3\le j\le g$, we compute the determinant to be
$$
\begin{aligned}
D_{12j}:=\det(\ddots)&=\det\left(
\begin{smallmatrix}
\frac12\phi_1\cdot X_2&1&\phi_j\tau_{2j}+\phi_j\sum\phi_\alpha\tau_{2\alpha}\tau_{j\alpha}\\
*&\frac12\phi_2\cdot (X_2+Y_2) &\phi_j\tau_{1j}+\phi_j\sum\phi_\alpha\tau_{2\alpha}\tau_{j\alpha}\\
*&*&\frac12\phi_j\cdot (X_2+Y_2)+\frac12\psi_j\tau_{1j}\tau_{2j}
\end{smallmatrix}\right)\\ &=
-\frac12\phi_j\cdot (X_2+Y_2)+(2\phi_j^2-\psi_j/2)\tau_{1j}\tau_{2j}+O(\ep^3)\,.
\end{aligned}
$$
We also recall the expansion of $\theta_{m_0}$ itself, given by equation~\eqref{eq:thetaexpand}. Using~\Cref{lm:cutexpand}, we will work with the expansions of the theta constants and the determinants of the $3\times 3$ minors of the Hessian up to $O(\ep^3)$, intersected with $Z$. We thus compute
$$
\theta_{m_0}|_Z=\left([1,2]+Y_2\right)|_Z+O(\ep^3)=\tfrac{1}{2\pi i}\tau_{12}+\tfrac{1}{(2\pi i)^2}\sum_{j\ge 3} \phi_j\tau_{1j}^2+O(\ep^3)\,.
$$
Substituting this into $D_{12j}$ and dropping the common and generically non-zero factor of $\prod f_\alpha$ gives
$$
\begin{aligned}
D_{12j}|_{Z\cap\tnmz}&=\left.-\tfrac12\phi_j\cdot (\tau_{12}+Y_2)+(2\phi_j^2-\tfrac{\psi_j}{2})\tau_{1j}\tau_{2j}\right|_{Z\cap\tnmz}+O(\ep^3)\\
&=O(\ep^3)+(2\phi_j^2-\psi_j/2)\tau_{1j}^2\,.
\end{aligned}
$$
Since the expression $2\phi_j^2-\psi_j/2$ is not identically zero in $t_j$, for a generic value of $t_j$ the vanishing of $D_{12j}|_{Z\cap\tnmz}$ implies $\tau_{1j}^2=O(\ep^3)$ and then substituting this back, the vanishing of $\theta_{m_0}|_Z$ implies $\tau_{12}=O(\ep^3)$, so that altogether we get precisely the vanishing of $\tau_{12},\tau_{13}=\tau_{23},\dots,\tau_{1g}=\tau_{2g}$ up to higher order, as required.
\end{proof}
\begin{proof}[Proof of{~\Cref{thm:tn2}}]
We observe that by the factorization of theta functions the big diagonal $\LL^e\subset \tnmz^2\subset \tnmz$. The above computation, using \Cref{lm:cutexpand}, shows that $\HH_1\times\HH_{g_1}$ is an irreducible component of $\tnmz^2$ containing $\II_g$, and since the defining equations of $\vtn^2$ are $\Gamma_g$ invariant, by \Cref{cor:maxcomponent} it follows that $\calA_1\times\calA_{g-1}$ is an irreducible component of $\vtn^2$.
\end{proof}

\subsection{The rank three locus}
We now deal with the locus $\vtn^3$; here our goal is to show that the locus of Jacobians with a vanishing theta-null is an irreducible component. Recall that the locus $\tn\chars\ep\de\cap\JJ_g$ is purely of dimension $3g-4$, and in fact by \cite{mosm}, both $\JJ_g\subset\HH_g$ and $\tn\chars\ep\de\subset\HH_g$ are irreducible; moreover, the intersection $\vtn\cap\calJ_g\subset\calA_g$ is irreducible by~\cite{teixidor}.
Thus we can apply \Cref{cor:maxcomponent}, so that it will again suffice to work with the Taylor expansions of the Hessian of $\theta_{m_0}$, and not with an arbitrary characteristic.

The additional complication in this case is that, for high genus, the dimension of $\vtn\cap\calJ_g$ is smaller than the dimension of a number of irreducible components of $\calR_g$ that are contained in $\vtn^3$ and contain $\calD_g$ and $\LL_g^o$ (for example, for the component $\calA_g\times\calA_{g-2}$).

We demonstrate this issue in detail for the various components of interest in low genus. For $g=5$ we have 
$$\dim(\calA_3\times \calA_2)=9<\dim(\calJ_5\cap\vtn)=11=\dim (\calA_4\times\calA_1)\,,$$ 
and indeed $\calA_3\times\calA_2$ is contained in the closure of $\calJ_5\cap\vtn$ since $\calJ_3=\calA_3$. However, already for genus 6 we have 
$$\dim(\calA_4\times\calA_2)=10<\dim(\calJ_6\cap\vtn)=14<\dim(\calA_5\times\calA_1)=16\,,$$
and thus $\calA_4\times\calA_2$ must be contained in an irreducible component of $\vtn^3$ that has dimension at least 14; moreover, note that $\calJ_6\cap\vtn\not\supset\calA_4\times\calA_2$. 

This discussion makes the following statement more surprising in that we can show that components of $\vtn^3$ not contained in the decomposable locus have expected dimension.

\begin{thm}\label{thm:ell=2}
For any genus $g$, the locus $\tnmz^3\setminus\RR_g$ locally near $\II_g$ has dimension $3g-4$.
\end{thm}
As a consequence, all irreducible components of $\vtn^3$ containing $\calD_g$ and not contained in the decomposable locus $\calR_g$ must have dimension $\leq 3g-4$.
\begin{cor}
For any genus $g\geq 3$, the locus $\vtn^3\setminus\calR_g $ locally near $\calD_g$ has dimension equal to $3g-4$.
\end{cor}
\begin{proof}
Indeed, we observe that $\LL_g^e\subset \tnmz\cap\JJ_g\subset \tnmz^3$. The locus $\vtn^3$ is by definition $\Gamma_g$-invariant, and hence the conditions of~\Cref{cor:maxcomponent} are satisfied.
\end{proof}
An immediate consequence is that the $(3g-4)$-dimensional irreducible locus $\calJ_g\cap \vtn=p(\tnmz\cap\JJ_g)$, contained in $\vtn^3$, is an irreducible component of $\vtn^3$. This finishes the proof of~\Cref{thm:tn3}, once we obtain the local dimension statement.
\begin{proof}[Proof of~\Cref{thm:ell=2}]
As above, we will be working in a sufficiently small neighborhood $U\supset\II_g$ of the diagonal, using the expansions~\eqref{eq:thetaexpand} for $\theta_{m_0}$ and the expansion~\eqref{hess2} for the $4\times 4$ minors of its Hessian. We first note that the vanishing of $\theta_{m_0}=X_2+Y_2+O(\ep^3)$ implies that $\tau_{12}=O(\ep^2)$ and moreover that $X_2+Y_2=O(\ep^3)$. 

Since we are interested in the locus where the rank of the Hessian is equal to 3, and not 2, and since the Hessian symmetric, there must exist a {\em principal} $3\times 3$ minor of the Hessian with a non-zero determinant. Notice that the $2\times 2$ principal minor of the Hessian formed by rows and columns $1$ and $2$ becomes, after plugging in $X_2+Y_2=O(\ep^3)$ from the vanishing of $\theta_{m_0}$, equal to $\left(\begin{smallmatrix} 
O(\ep^3)&1+O(\ep^2)\\ 1+O(\ep^2)&O(\ep^3)\end{smallmatrix}\right)$, so that its determinant is equal to $-1+O(\ep^2)$, and thus non-zero. Since this is a non-degenerate principal $2\times 2$ minor, for the matrix to have rank equal to $3$, it must be contained in a non-degenerate $3\times 3$ minor. Moreover, since the matrix is symmetric, there must exists a principal such non-degenerate $3\times 3$ minor, and by renumbering the coordinates, we thus assume without loss of generality that this non-degenerate minor is made up by rows and columns number $1,2,3$. 
Similarly to the proof of \Cref{thm:tn2}, for convenience we will intersect, in a neighborhood $U\supset\II_g$, the locus $\tnmz^3$, with the codimension $g-2$ subvariety $Z\subset\HH_g$ given by equations $\tau_{1j}=\tau_{2j}$ for all $3\le j\le g$. Our goal is to prove that $U\cap Z\cap\tnmz^3$ has dimension at most $(3g-4)-(g-2)=2g-2$. Indeed, since we know that $\JJ_g\cap\tnmz)$ has dimension precisely $3g-4$, and thus $\dim (U\cap\JJ_g\cap\tnmz\cap Z)\ge 2g-2$ is contained in $\tnmz^3\cap Z$, this will then imply \Cref{thm:ell=2}.

To bound from above the dimension of $U\cap Z\cap\tnmz^3$, we will look at determinants $D_{jk}$ of the $4\times 4$ minors of the Hessian made up of rows $(123j)$ and columns $(123k)$, for any $4\le j\le k\le g$. We will see that on $U\cap Z\cap\tnmz$, for $j=k$ the vanishing of $D_{jj}$ will require $\tau_{1j}=\tau_{2j}$ to vanish to higher order, and then we will see that the vanishing of $D_{jk}$ for $j<k$ determines $\tau_{jk}$ in terms of the other variables, up to higher order. Thus altogether, up to higher order, the point of $U\cap Z\cap\tnmz^3$ will be determined by the values of the diagonal period matrix elements $t_1=\tau_{11},\dots,t_g=\tau_{gg}$, together with $\tau_{13}=\tau_{23}$, and $\tau_{34},\dots,\tau_{3g}$ (recall that $\tau_{12}$ is determined in terms of other coordinates, up to higher order, from the vanishing of $\theta_{m_0}$. Thus altogether by applying \Cref{lm:cutexpand}, we will see that the dimension of $U\cap Z\cap\tnmz^3$ is equal the number of these coordinates, i.e.~$g+1+(g-3)=2g-2$. We now inspect these $4\times 4$ minors in detail.

First, recall from the proof of \Cref{thm:tn2} that the determinant $D_{123}$ of the $3\times 3$ minor of the Hessian formed by the first 3 rows and columns is equal, up to higher order terms and generically non-vanishing factor, to $\tau_{13}^2$. Since $D_{123}\ne 0$ by assumption, this means that $\tau_{13}\ne 0$.

Now, for a principal $4\times 4$ minor $D_{jj}$, we plug in $X_2+Y_2=O(\ep^3)$ into~\eqref{hess2}, and see that the lowest order entries of the minor are as follows:
$$
\left(
\begin{smallmatrix}
O(\ep^3)&1+O(\ep^2)&\phi_3\cdot[2,3]+O(\ep^2)&\phi_j\cdot[2,j]+O(\ep^2)\\
*&O(\ep^3) &\phi_3\cdot[1,3]+O(\ep^2)&\phi_j\cdot[1,j]+O(\ep^2)\\
*&*&\frac12\psi_3\cdot[3,3,1,2]+O(\ep^3)&\phi_3\cdot\phi_j\cdot[3,j,1,2]+\dots\\
*&*&*&\frac12\psi_j\cdot[j,j,1,2]+O(\ep^3)
\end{smallmatrix}\right)\,.
$$
Thus the lowest order term that could appear in the determinant of this matrix is of order $O(\ep^4)$, and we write it explicitly in terms of the entries of the period matrix (noting, importantly, that in $[j,k,1,2]$ the term $\tau_{jk}\tau_{12}$ is higher order, and using Maple to compute safely)
\begin{equation}\label{eq:ep4}
 D_{12jk}=-\tau_{13}^2\tau_{1j}^2(4\phi_3^2-\psi_3)(4\phi_j^2-\psi_j)/4+O(\ep^5)
\end{equation}
For generic values of $t_3,t_j$ the expressions depending on them are non-zero, and thus the vanishing of such a determinant implies, since $\tau_{13}\ne 0$, that $\tau_{1j}=0$, up to higher order terms.

We now inspect the determinant of $D_{jk}$ of the $4\times 4$ minor formed by rows $(123j)$ and columns $(123k)$ for $j<k$; all the terms can again be read off from~\eqref{hess2}, so that the corresponding $4\times 4$ minor is as follows (where to make the formula fit we dropped the $1/2\pi i$ factors in front of each $\tau$, coming from the bracket expressions, and we recalled $\tau_{1a}=\tau_{2a}$ for $a=3,j,k$). Note that the fourth row and column of the minor are no longer symmetric.
$$
\left(
\begin{smallmatrix}
O(\ep^3)&1+O(\ep^2)&\phi_3\tau_{13}+O(\ep^2)&\phi_k\tau_{1k}+O(\ep^2)\\
1+O(\ep^2)&O(\ep^3) &\phi_3\tau_{13}+O(\ep^2)&\phi_k\tau_{1k}+O(\ep^2)\\
\phi_3\tau_{13}+O(\ep^2)&\phi_3\tau_{13}+O(\ep^2)&\frac12\psi_3\cdot[3,3,1,2]+O(\ep^3)&\phi_3\cdot\phi_k\cdot[3,k,1,2]+\dots\\
\phi_j\tau_{1j}+O(\ep^2)&\phi_j\tau_{1j}+O(\ep^2)&\phi_3\cdot\phi_j\cdot[3,j,1,2]+\dots&\phi_j\cdot\phi_k\cdot[j,k,1,2]+\dots
\end{smallmatrix}\right)\,.
$$
Notice, however, that this formula does not really give the lowest order terms of the expansion, as indeed by the vanishing of the determinants $D_{jj}$ and $D_{kk}$ of the principal minors we know that $\tau_{1j},\tau_{1k}=O(\ep^2)$. Thus in fact the entries $(1,k),(2,k),(j,1),(k,1)$ of the minor containing these entries are themselves of order $O(\ep^2)$, while the correction term to $\phi_j\tau_{1j}$ is actually of higher order, as all the brackets involved will contain $\tau_{1j}$ or $\tau_{12}$, and are thus of order at least one higher than their degree in $\tau$'s. 

What we want to determine is the dependence of $D_{jk}$ on $\tau_{jk}$, more precisely we want to determine the lowest order term that contains $\tau_{jk}$. By inspection, we see that 
$$
 [j,k,1,2]=\tau_{jk}\tau_{12}+2\tau_{1j}\tau_{2k}=\tau_{jk}\tau_{12}+O(\ep^4)
$$
appearing in the $(j,k)$ entry of the minor above is the only entry where $\tau_{jk}$ appears. We recall that from the vanishing of $\tnmz$, given by expansion~\eqref{eq:thetaexpand}, we have (again, up to all the $\pm 2\pi i$ factors) 
$$
X_2+Y_2=O(\ep^3)=\tau_{12}+\sum_{a>2}\phi_a\tau_{1a}\tau_{2a}=\tau_{12}+\tau_{13}^2+\sum_{j>3}\tau_{1j}^2=\tau_{12}+\tau_{13}^2+O(\ep^4)\,,
$$
since $\tau_{1j}=O(\ep^2)$. Thus we see that $\tau_{12}=-\tau_{13}^2+O(\ep^3)$, and is of order precisely $\ep^2$, as $\tau_{13}$ is non-zero due to the assumed non-vanishing of the determinant $D_{123}$. Thus the $(j,k)$ entry of the minor above contributes $\phi_j\phi_k(-\tau_{jk}\tau_{13}^2+O(\ep^4))\cdot D_{123}$ to $D_{jk}$, where we expanded $D_{jk}$ using the last row. By assumption $D_{123}$ is non-zero, and in fact of order $O(\ep^2)$ as discussed above. By inspection of the minor, the only other places where $\tau_{jk}$ appears in the minor are when expanding brackets of 6 terms, and then at least two of these terms would be of order $O(\ep^2)$, so we have found that the only dependence of $D_{jk}$ modulo $O(\ep^6)$ on $\tau_{jk}$ is the term $-\phi_j\phi_k\tau_{jk}\tau_{13}^2\cdot D_{123}$. Thus requiring $D_{jk}$ to vanish modulo $O(\ep^6)$ expresses $\tau_{jk}$ in terms of the other variables, modulo $O(\ep^2)$.

Thus altogether each $\tau_{1j}=\tau_{2j}$ must be of order $O(\ep^2)$ by the vanishing of $D_{jj}$, while each $\tau_{jk}$ is expressed in terms of the rest of the entries in the first 3 rows of the period matrix, and the diagonal entries, by computing the $O(\ep^5)$ term of $D_{jk}$ and requiring it to vanish. Altogether we see that the local dimension of the locus $\tnmz^3\cap U\cap Z$ is as claimed.
\end{proof}


\begin{thebibliography}{FGSM21}

\bibitem[AC21]{agch}
D.~Agostini and L.~Chua.
\newblock On the {S}chottky problem for genus-five {J}acobians with a vanishing
 theta-null.
\newblock {\em Ann. Sc. Norm. Super. Pisa Cl. Sci. (5)}, 22(1):333--350, 2021.

\bibitem[Far06]{farkashvanishing}
H.~Farkas.
\newblock Vanishing thetanulls and {J}acobians.
\newblock In {\em The geometry of {R}iemann surfaces and abelian varieties},
 volume 397 of {\em Contemp. Math.}, pages 37--53. Amer. Math. Soc.,
 Providence, RI, 2006.

\bibitem[FGSM21]{fagrsm}
H.~Farkas, S.~Grushevsky, and R.~Salvati~Manni.
\newblock An explicit solution to the weak {S}chottky problem.
\newblock {\em Algebr. Geom.}, 8(3):358--373, 2021.

\bibitem[Fre68]{freitagF}
E.~Freitag.
\newblock Fortsetzung von automorphen {F}unktionen.
\newblock {\em Math. Ann.}, 177:95--100, 1968.

\bibitem[Fre91]{freitagbooksingular}
E.~Freitag.
\newblock {\em Singular modular forms and theta relations}, volume 1487 of {\em
 Lecture Notes in Mathematics}.
\newblock Springer-Verlag, Berlin, 1991.

\bibitem[GH12]{grhu1}
S.~Grushevsky and K.~Hulek.
\newblock The class of the locus of intermediate {J}acobians of cubic
 threefolds.
\newblock {\em Invent. Math.}, 190(1):119--168, 2012.

\bibitem[GSM07]{grsmordertwo}
S.~Grushevsky and R.~Salvati~Manni.
\newblock Singularities of the theta divisor at points of order two.
\newblock {\em Int. Math. Res. Not. IMRN}, (15):Art. ID rnm045, 15, 2007.

\bibitem[GSM08]{grsmgen4}
S.~Grushevsky and R.~Salvati~Manni.
\newblock Jacobians with a vanishing theta-null in genus 4.
\newblock {\em Israel J. Math.}, 164:303--315, 2008.

\bibitem[GSM09]{grsmconjectures}
S.~Grushevsky and R.~Salvati~Manni.
\newblock The loci of abelian varieties with points of high multiplicity on the
 theta divisor.
\newblock {\em Geom. Dedicata}, 139:233--247, 2009.

\bibitem[Igu72]{igusabook}
J.-I. Igusa.
\newblock {\em Theta functions}, volume 194 of {\em Grundlehren der
 Mathematischen Wissenschaften}.
\newblock Springer-Verlag, New York, 1972.

\bibitem[Igu80]{igusajacobi}
J.-I. Igusa.
\newblock On {J}acobi's derivative formula and its generalizations.
\newblock {\em Amer. J. Math.}, 102(2):409--446, 1980.

\bibitem[Mah69]{mahler}
K.~Mahler.
\newblock On algebraic differential equations satisfied by automorphic
 functions.
\newblock {\em J. Austral. Math. Soc.}, 10:445--450, 1969.

\bibitem[MSM21]{mosm}
G.~Mondello and R.~Salvati~Manni.
\newblock Totally irreducible subvarieties of {S}iegel moduli spaces.
\newblock 2021.

\bibitem[Mum83]{mumforddimag}
D.~Mumford.
\newblock On the {K}odaira dimension of the {S}iegel modular variety.
\newblock In {\em Algebraic geometry---open problems ({R}avello, 1982)}, volume
 997 of {\em Lecture Notes in Math.}, pages 348--375, Berlin, 1983. Springer.

\bibitem[Mum07]{mumfordbooktheta2}
D.~Mumford.
\newblock {\em Tata lectures on theta. {II}}.
\newblock Modern Birkh\"auser Classics. Birkh\"auser Boston Inc., Boston, MA,
 2007.
\newblock Jacobian theta functions and differential equations, With the
 collaboration of C. Musili, M. Nori, E. Previato, M. Stillman and H. Umemura,
 Reprint of the 1984 original.

\bibitem[Poo94]{poor}
C.~Poor.
\newblock The hyperelliptic locus.
\newblock {\em Duke Math. J.}, 76(3):809--884, 1994.

\bibitem[SB21]{sbpoincare}
N.~I. Shepherd-Barron.
\newblock Asymptotic period relations for {J}acobian elliptic surfaces.
\newblock {\em Proc. Lond. Math. Soc. (3)}, 122(4):479--520, 2021.

\bibitem[SM94]{smlevel2}
R.~Salvati~Manni.
\newblock Modular varieties with level {$2$} theta structure.
\newblock {\em Amer. J. Math.}, 116(6):1489--1511, 1994.

\bibitem[TiB88]{teixidor}
M.~Teixidor~i Bigas.
\newblock The divisor of curves with a vanishing theta-null.
\newblock {\em Compositio Math.}, 66(1):15--22, 1988.

\bibitem[Tsu91]{tsu}
Sh. Tsuyumine.
\newblock Thetanullwerte on a moduli space of curves and hyperelliptic loci.
\newblock {\em Math. Z.}, 207(4):539--568, 1991.

\end{thebibliography}

\end{document}